\documentclass[11pt]{amsart}

\usepackage{moreverb,url}

\usepackage{amssymb}
\usepackage{amsmath}
\usepackage{amsthm}
\usepackage{amsfonts} 

\usepackage{mathrsfs}
\usepackage{mathtools}
\usepackage{dsfont}
\usepackage{esint}
\usepackage[normalem]{ulem}
\usepackage{marginnote}
\usepackage{array}
\usepackage{soul}
\usepackage{comment}

\usepackage[utf8]{inputenc}  
\usepackage[T1]{fontenc}  

\usepackage{booktabs} 

\usepackage{MnSymbol}

\usepackage{xpatch}

\NewCommandCopy{\subsectionnotoc}{\subsection}
\xpatchcmd{\subsectionnotoc}{{2}}{{1001}}
  {}
  {\PackageError{main}{Patch fuer subsectionnotoc fehlgeschlagen}{}}

\usepackage{longtable}

\usepackage[all]{xy}

\usepackage[english]{babel} 
\usepackage[text={170.5mm,226mm},%
papersize={210mm,280mm},%
columnsep=12pt,%
headsep=21pt,%
centering]{geometry}
\usepackage[
    backend=biber,
    style=numeric,
    url=false,
    doi=true,
    isbn=false
]{biblatex}

\usepackage{xargs}

\usepackage{csquotes}
\usepackage{stmaryrd}

\usepackage{enumitem}

\usepackage[colorlinks,bookmarksopen,bookmarksnumbered,citecolor=blue,urlcolor=blue]{hyperref}
\usepackage[nameinlink]{cleveref}

\Crefname{assumption}{Assumption}{Assumptions}
\Crefname{counterexample}{Counterexample}{Counterexamples}
\Crefname{example}{Example}{Examples}
\Crefname{definition}{Definition}{Definitions}
\Crefname{section}{Section}{Sections}
\Crefname{lemma}{Lemma}{Lemmata}
\Crefname{proposition}{Proposition}{Propositions}
\Crefname{theorem}{Theorem}{Theorems}
\Crefname{corollary}{Corollary}{Corollaries}
\Crefname{remark}{Remark}{Remarks}
\Crefname{known}{Known Result}{Known Results}
\Crefname{caveat}{Caveat}{Caveats}

\theoremstyle{plain}
\newtheorem{theorem}{Theorem}[section]
\newtheorem{corollary}[theorem]{Corollary}
\newtheorem{proposition}[theorem]{Proposition}
\newtheorem{lemma}[theorem]{Lemma}
\newtheorem{known}[theorem]{Known Result}
\theoremstyle{definition}
\newtheorem{definition}[theorem]{Definition}

\theoremstyle{remark}
\newtheorem{remark}[theorem]{Remark}
\newtheorem{example}[theorem]{Example}

\newtheorem{caveat}[theorem]{Caveat}

\theoremstyle{definition}
\theoremstyle{remark}

\numberwithin{equation}{section}

\AddToHook{env/theorem/begin}
  {\crefalias{theorem}{theorem}}

\AddToHook{env/corollary/begin}
  {\crefalias{theorem}{corollary}}

\AddToHook{env/proposition/begin}
  {\crefalias{theorem}{proposition}}

\AddToHook{env/lemma/begin}
  {\crefalias{theorem}{lemma}}

\AddToHook{env/known/begin}
  {\crefalias{theorem}{known}}

\AddToHook{env/definition/begin}
  {\crefalias{theorem}{definition}}

\AddToHook{env/assumption/begin}
  {\crefalias{theorem}{assumption}}

\AddToHook{env/caveat/begin}
  {\crefalias{theorem}{caveat}}

\newcommand{\R}{\mathbb{R}}
\newcommand{\N}{\mathbb{N}}
\newcommand{\NN}{\mathbb{N}}

\mathchardef\emptyset="001F

\newcommand{\C}{\mathbb{C}}

\newcommand{\Baire}{\mathcal{B}}
\newcommand{\dom}{\mathrm{dom}}

\newcommand{\Lim}{\operatorname{lim}}

\newcommand{\SCIG}{\mathrm{SCI}_G}
\newcommand{\SCIA}{\mathrm{SCI}_A}

\title{\bf Foundational Analysis Of The Solvability Complexity Index}

\begin{document}

\author[C.~Sorg]{Christopher Sorg$^1$}
	\address[C.~Sorg]{
		\textup{Chair for Theoretical computer science, mathematics, and operations research}
		\newline \indent
		\textup{Department of Computer Science} \newline \indent
		\textup{University of the Bundeswehr Munich}
		\newline \indent
		\textup{85577 Neubiberg, Germany}}

\email{{\href{mailto:chr.sorg@unibw.de}{\textcolor{blue}{\texttt{chr.sorg@unibw.de}}}}
}

\footnotetext[1]{Inf1, University of the Bundeswehr Munich, Werner-Heisenberg-Weg 39, 85577 Neubiberg, Germany}

\begin{abstract}
\noindent The Solvability Complexity Index (SCI) provides an extensional limit-height formalism for recovering a target map \(\Xi\) from finite samples of an evaluation interface \(\Lambda\subseteq\mathbb C^\Omega\) by finite-height towers of pointwise limits. We first give a foundational analysis of what this extensional framework does and does not determine. In second place, we give a concrete non-trivial bridge result from a comparable SCI object to the well known limit rank of Weihrauch theory.
\end{abstract}

\maketitle

\begin{center}\small
\textbf{Keywords:} Solvability Complexity Index (SCI); computable analysis; Weihrauch reducibility; iterated limits; Borel hierarchy
\end{center}
\tableofcontents

\section{Introduction}

The \textit{Solvability Complexity Index} (SCI) was introduced as a classification tool for
mathematical tasks that are not naturally captured by a single-shot notion of computation,
but instead admit (or require) \textit{towers of algorithms}: procedures whose outputs converge
to the desired value only after one or several limiting operations \cite{hansen2011solvability,ben2015computing,ColbrookHansen22}.
This perspective has proved particularly fruitful in computational spectral theory, where
the existence (or impossibility) of algorithms can depend delicately on the permitted input
interface and on the allowed post-processing of oracle information \cite{hansen2011solvability,ColbrookHansen22}. Studying the classified computational spectral problems by the SCI literature we fastly got the intuition that infinite dimensional spectral approximation problems are not merely the motivation for the concept of the SCI, also the SCI definition \textit{itself} seems natural for such infinite dimensional problems. And this point can easily be explained mathematical: Assume that every input point can be identified by finitely many observations, meaning
\[
	\bigl( \forall A\in\Omega \bigr) \bigl( \exists Q_A\subseteq\Lambda \bigr)
\bigl( 0<|Q_A|<\infty \ \land\ \{B\in\Omega:(\forall\lambda\in Q_A)\, \lambda(B)=\lambda(A)\}=\{A\} \bigr),
\]
where $\Omega$ is a set (input class), $\Lambda \subseteq \mathbb{C}^{\Omega}$ (evaluation set) and $Q_A$ is called a Query. It is then easy (at least once studied the precise definitions following in \Cref{sec:SCIdef}) to see that every $\Xi : \Omega \to \mathcal M$ for a metric space $\mathcal M$ has $\SCIG$ value $0$. This identifies: For e.g. $\Omega \subseteq \mathbb{R}^d$ with complete knowledge about the coordinates, the $d$ Queries identify the input, so we have $\SCIG$ value $0$ for every target map $\Xi$. In the case of \textit{infinite} sequences, functions or operator(matrices) we don't have this full identification of the input in general. It is therefore possible here to have non-trivial $\SCIG$ values.

It is important to note that the SCI is intentionally \textit{extensional}: at its core, one fixes an input class $\Omega$, a target map $\Xi:\Omega\to \mathcal{M}$ for a metric space $\mathcal{M}$, and an evaluation family $\Lambda$, and defines what it means for a tower to compute $\Xi$ from finitely many evaluation values (possibly iterated through limits). Here, extensional is not the set-theoretic axiom of extensionality, the corresponding SCI-definitions just don't fix (on purpose) any syntax-level notation of computation yet.

This abstraction, as also claimed explicitly in e.g. \cite[Slide~11]{ColbrookPopov}, is a strength: it allows one to compare disparate numerical and information
models under a single roof but it also raises a foundational question that motivates the
present work:
\begin{align*}
	\textbf{Q:} \text{ When does the SCI define a computability model?}
\end{align*}
A \enquote{computability model} is not merely a set of input-output tasks; it is a stable notion of
\textit{morphism} closed under composition and basic data operations, so that procedures can be combined, relativized, and compared without ambiguity.
In Type-2 computable analysis this stability is enforced by \textit{representations} (names) and
\textit{realizers}: maps between represented spaces form a robust category in which computable
operations are automatically continuous at the level induced by the representations
\cite{weihrauch2000computable,brattka2021weihrauch}.
Weihrauch reducibility is meaningful precisely because it lives inside such a compositional
setting: it compares problems by uniform pre-/post-processing around a single oracle call,
and the resulting degree structures support a rich algebra of computational content
\cite{brattka2011weihrauch,brattka2021weihrauch}.

In contrast, the purely extensional SCI definition leaves open what counts as admissible
post-processing on finitely many oracle values.
Different interpretations of this interface can lead to different \enquote{heights} for the
same extensional task.
A classical example of such divergence is the gap between \textit{atomic} real-number interfaces
(which allow exact comparisons) and \textit{name-based} interfaces (Type-2), where computability
is necessarily continuous \cite{weihrauch2000computable,neumann2018topological,BSS98}.

To precise and honor the history of development of related concepts, we want to shortly also add: Effective Borel measurability and realizer reducibilities were already studied in \cite{brattka05}; for the Weihrauch framework, see \cite{brattka2011weihrauch}. Parallel to this work (which first version was uploaded in March 2026 on arXiv) Wyeth and Hutter \cite{Wyeth26} developed arithmetic hierarchies of real-valued functions through uniform limiting procedures. Their real-domain characterizations in Section~6 adopt additional effective semimeasurability requirements, where the present work concerns the specific distinction between exact finite-query dependence and uniform realizability of SCI towers. We will reference inside this work to their results more precise where suitable.

\textbf{Main results.}$\,$ Apart from sharpening and necessarily extending basic SCI definitions and concepts, this work has three core parts. First, the raw finite-information
definition yields a positive-height Baire-class characterization for countable interfaces (\Cref{prop:lowerboundSCI}); \Cref{ex:uncountable-counterexample} shows why the countability
assumption cannot simply be dropped. Second, the alternating Cantor predicates realize every finite raw height (\Cref{lem:alternating-calibration}); their tagged union has height $\infty$ (\Cref{thm:Xi-infty-height}), which deals as an example of an unsolvability/impossibility problem. Third, after fixing names for the information table, the standard limit-normal-form machinery gives a concrete name-tower characterization (\Cref{thm:NTNF}). 

The separation examples then identify the additional restrictions needed to relate this effective notion to raw finite-information towers; in particular, \Cref{prop:individual-not-uniform} distinguishes individual computability
from a uniform realization of the indexed family.

Furthermore, as general algorithm contains the word algorithm, we analyzed in \Cref{sec:AlgGA} typical algorithm-like algebraic properties like currying, existence of universal algorithms to understand if they also stand for general algorithms. Here, we mean by algorithm "effective procedures" formalized by Turing in \cite[§§1-2]{Turing36} and examine finite-program universality as in \cite[§§6-7]{Turing36}.

Here is an overview of the SCI concepts appearing in this work.

\begin{center}
	\small
	\begin{tabular}{p{0.23\linewidth}p{0.46\linewidth}p{0.22\linewidth}}
		\hline
		Notation & Meaning and data that must be fixed & Defined or used in \\
		\hline
		$\SCIG$
		& Stable finite-query leaves; arbitrary exact post-processing; metric output limits
		& \Cref{def:standard-tower}, \Cref{def:sci} \\
		$\SCIA$
		& One parameter-free BSS oracle machine for all indices; fixed enumeration and finite output decoder
		& \Cref{def:BSS-arith-alg}, \Cref{def:sci} \\
		$\operatorname{SCI}_{\mathrm{name}}$
		& One Type-2 computable array of Baire streams; only the final limit must realize the represented target
		& \Cref{def:nameTower}, \Cref{def:name-sci} \\
		$\operatorname{rank}_{\rm W}(f)$
		& Least $k$ such that $f\leq_{\rm W} L_k$; representations fixed
		& \Cref{def:lim-rank} \\
		$\operatorname{SCI}^{\mathrm{all}}_G$
		& Auxiliary variant requiring every non-final intermediate map to be a general algorithm
		& \Cref{def:all-level-tower} \\
		\hline
	\end{tabular}
\end{center}

\section{The Standard Extensional SCI And Implemented Readings}{\label{sec:SCIdef}}

We are interested in measuring the \enquote{algorithmic complexity} of mathematical spectral approximation problems in the sense of the \textit{Solvability Complexity Index} (SCI), which measures the minimal number of nested limits required by any algorithm to solve a given (infinite-dimensional) problem \cite{hansen2011solvability,ben2015computing,ben2015new,ColbrookHansen22}. The SCI framework has revealed sharp boundaries for spectral computations of (unbounded) operators. Some spectral problems are provably not solvable by any single-limit procedure (hence admit no global error control), yet become solvable by towers of algorithms of finite height or may not be solvable for a finite height at all, in the latter the SCI being therefore infinity, see e.g.\ \cite[Thm.~B.1]{colbrook2024limits}. We first state the standard extensional SCI. We then distinguish it from implemented variants, where the indexed approximation table is required to be produced uniformly by a specified computation model. This distinction is logically necessary for Weihrauch comparisons.

\textbf{Notation:} In this work, we will always use the notation $\mathbb N:=\{1,2,\ldots\}$, $\mathbb N_0:=\mathbb N\cup\{0\}$, and $2:=\{0,1\}$.
The set $2$ carries the discrete topology and metric. We put further
\[
 \mathcal B:=\mathbb N_0^{\mathbb N_0},\quad
 \mathcal C:=2^{\mathbb N_0},\quad
 d_{\mathcal C}(x,y):=\sum_{j=0}^{\infty}2^{-j-1}|x(j)-y(j)|.
\]
The Baire space $\mathcal B$ and Cantor space $\mathcal C$ carry their product topologies. We write $0^\omega\in\mathcal C$ for the constant-zero sequence, and put
\[
 x^{[n]}(j):=
 \begin{cases}x(j),&j<n,\\0,&j\geq n\end{cases}
 \quad(n\in\mathbb N_0).
\]

\begin{definition}[Metric space]\label{def:metric}
	A \textit{metric space} is a pair $(\mathcal{M},d)$ where $\mathcal{M}$ is a set and the map
	$d:\mathcal{M}\times\mathcal{M}\to [0,\infty)$ satisfies $\forall x,y,z \in \mathcal{M}$:
	(i) $d(x,y)=0 \Leftrightarrow x=y$;
	(ii) $d(x,y)=d(y,x)$;
	(iii) $d(x,z)\le d(x,y)+d(y,z)$.
\end{definition}

\begin{definition}[Complex-valued evaluation set]\label{def:eval-set}
	Let $\Omega$ be a set. Write $\C^\Omega:=\{f:\Omega\to\C\}$.
	An \textit{evaluation set} on $\Omega$ is a subset $\Lambda\subseteq \C^\Omega$.
	Equivalently, $\Lambda$ is a family of maps $f:\Omega\to\C$.
\end{definition}

\begin{remark}[Non-$\C$-valued \enquote{natural} evaluations]\label{rem:nonC}
	Many evaluations are not $\C$-valued.
	In the sense of \Cref{def:eval-set}: If an evaluation takes values in a set $V$ with $|V|\leq|\mathbb C|$, one may choose an injection $\iota:V\hookrightarrow\mathbb C$ and replace $g:\Omega\to V$ by $\iota\circ g$.
	Since
	\[
		\iota(g(A))=\iota(g(B))\iff g(A)=g(B),
	\]
	this recoding preserves exact finite-observation information and the type-$G$ model. Attention is needed for preserving topological regularity or implemented computability. Such preservation(s) requires corresponding assumptions on the coding and decoding maps. For the main results of this work, the latter is no loss of generality, so we decided to stay with this complex-valued convention to simplify the presentation of the work.
\end{remark}

The now following SCI definitions are based on the definitions in \cite[App.~A]{colbrook2024limits} but slightly precised for our purposes in this work.

\begin{definition}[Computational problem; consistency \enquote{axiom}]\label{def:comp-prob}
	A \textit{computational problem} is a quadruple
	\[
	\mathcal{P}=\bigl( \Xi,\Omega,(\mathcal{M},d),\Lambda \bigr),
	\]
	where
	\begin{itemize}
		\item $\Omega$ is a set (the \textit{primary set}, i.e.\ input class);
		\item $(\mathcal{M},d)$ is a metric space in the sense of \Cref{def:metric};
		\item $\Xi:\Omega\to\mathcal{M}$ is a function (the \textit{problem function});
		\item $\Lambda\subseteq \C^\Omega$ is an evaluation set in the sense of \Cref{def:eval-set}.
	\end{itemize}
	We impose the \textit{consistency \enquote{axiom}}:
	\begin{equation}\label{eq:consistency}
		\forall A,B\in\Omega \, \bigl( \Xi(A)\neq \Xi(B)\ \Longrightarrow\ \exists f\in\Lambda\ \text{s.t. } f(A)\neq f(B) \bigr).
	\end{equation}
\end{definition}

For any set $X$, we will write in the following
\[
[X]^{<\omega} := \{E \subseteq X : E \text{ is finite}\}.
\]

\begin{definition}[Evaluation table and observational indistinguishability]\label{def:evaluation-table}
	Let $\Omega$ be a set and let
	$\varnothing\neq\Lambda\subseteq\mathbb C^\Omega$.
	Fix an indexing $(f_i)_{i\in J}$ whose range is $\Lambda$.
	For an arbitrary family one may use $J=\Lambda$ and $f_\lambda=\lambda$.
	When $\Lambda$ is countable, we fix once and for all a surjective
	sequence $(f_i)_{i\in\mathbb N}$; repetitions are permitted.
	Define
	\[
	\operatorname{Ev}_\Lambda:\Omega\to\mathbb C^J,\quad
	\operatorname{Ev}_\Lambda(A):=(f_i(A))_{i\in J},\quad
	I_\Lambda:=\operatorname{Ev}_\Lambda(\Omega).
	\]
	Thus $I_\Lambda\subseteq\mathbb C^{\mathbb N}$ in the countable case.
	Define
	\[
	A\sim_\Lambda B \iff
	(\forall\lambda\in\Lambda)\,\lambda(A)=\lambda(B).
	\]
	Surjectivity of the indexing implies that
	$A\sim_\Lambda B$ is equivalent to
	$\operatorname{Ev}_\Lambda(A)=\operatorname{Ev}_\Lambda(B)$.
\end{definition}

The following factorization is a simple but very important direct property of the evaluation set. Its role will be more clear when used later for the SCI-Weihrauch comparison.

\begin{proposition}[Factorization through the evaluation table]\label{prop:factorization}
	Let $\Xi:\Omega\to \mathcal M$ and
	$\varnothing\neq\Lambda\subseteq\mathbb C^\Omega$ be arbitrary data;
	consistency is not assumed for this statement. The following are
	equivalent:
	\begin{enumerate}
		\item For all $A,B\in\Omega$, if $\Xi(A)\neq\Xi(B)$, then
		$\lambda(A)\neq\lambda(B)$ for some $\lambda\in\Lambda$.
		\item $\Xi$ is constant on every $\sim_\Lambda$-equivalence class.
		\item There is a unique map $\Phi:I_\Lambda\to \mathcal M$ such that
		\[
		\Xi=\Phi\circ\operatorname{Ev}_\Lambda.
		\]
	\end{enumerate}
\end{proposition}

\begin{proof}
	The first two conditions are equivalent by contraposition.
	Under the second condition, the prescription
	\[
	\Phi(\operatorname{Ev}_\Lambda(A)):=\Xi(A)
	\]
	is well-defined. It gives the required factorization, and uniqueness
	follows from surjectivity of
	$\operatorname{Ev}_\Lambda:\Omega\to I_\Lambda$.
	Conversely, a factorization forces equal evaluation tables to have
	equal target values, which is the second condition.
\end{proof}

\begin{definition}[General algorithm]\label{def:gen-alg}
	Let $\mathcal{P}=\bigl( \Xi,\Omega,(\mathcal{M},d),\Lambda \bigr)$ be a computational problem.
	A \textit{general algorithm} for $\mathcal{P}$ is a pair $(\Gamma,\Lambda_\Gamma)$ where
	\[
	\Gamma:\Omega\to \mathcal{M}
	\quad\text{and}\quad
	\Lambda_\Gamma:\Omega\to [\Lambda]^{<\omega}
	\]
	assigns to each $A\in\Omega$ a finite nonempty set of queries $\Lambda_\Gamma(A)\subseteq\Lambda$
	such that for all $A,B\in\Omega$:\hfill
	\begin{enumerate}
		\item (\textbf{Finite-information dependence}): If $f(B)=f(A)$ for all $f\in\Lambda_\Gamma(A)$,
		then $\Gamma(B)=\Gamma(A)$.
		\item (\textbf{Stability of the query set}): If $f(B)=f(A)$ for all $f\in\Lambda_\Gamma(A)$,
		then $\Lambda_\Gamma(B)=\Lambda_\Gamma(A)$.
	\end{enumerate}
\end{definition}

\textbf{Convention.} Throughout this work, we always assume $\Lambda \neq \emptyset$ (if not stated otherwise explicitly) to avoid being there possible distracting philosophical discussions at some points. When $\Lambda$ is countable, an enumeration means a surjective sequence; repetitions are permitted. 

\begin{definition}[Fixed and input-dependent query policies]\label{def:query-policy}
	Let $(\Gamma,\Lambda_\Gamma)$ be a general algorithm in the
	sense of \Cref{def:gen-alg}. The map
	\[
	\Lambda_\Gamma:
	\Omega\rightarrow
	[\Lambda]^{<\omega}\setminus\{\varnothing\}
	\]
	is called its \textit{query policy}.
	
	\begin{enumerate}
		\item
		The pair $(\Gamma,\Lambda_\Gamma)$ is called
		\textit{fixed-query} if there exists
		\[
		Q\in[\Lambda]^{<\omega}\setminus\{\varnothing\}
		\]
		such that
		\[
		\Lambda_\Gamma(A)=Q
		\qquad\text{for every }A\in\Omega.
		\]
		
		\item
		The pair $(\Gamma,\Lambda_\Gamma)$ is called
		\textit{input-dependent-query}, or \textit{adaptive-query
			in the extensional sense}, if
		\[
		\exists A,B\in\Omega:
		\Lambda_\Gamma(A)\ne\Lambda_\Gamma(B).
		\]
	\end{enumerate}
\end{definition}

More simply spoken, "adaptive-query" here means exactly that $\Lambda_\Gamma$ is nonconstant. The distinction of these two query types is important as the following simple example illustrates.

\begin{example}[A stable policy need not be a sequential query trace]\label{ex:stable-policy-not-sequential}
	Let
	\[
	\Omega=\mathcal M=2,\quad
	\Lambda=\{\lambda_0,\lambda_1\},
	\]
	where
	\[
	\lambda_0(A):=A,\quad
	\lambda_1(A):=1-A.
	\]
	Define
	\[
	\Gamma(A):=A,\quad
	\Lambda_\Gamma(A):=\{\lambda_A\}.
	\]
	Each $\lambda_A$ is injective. Therefore agreement with $A$ on $\Lambda_\Gamma(A)$ forces equality with $A$, and both clauses of \Cref{def:gen-alg} hold.
	
	Consider a deterministic sequential oracle procedure whose only
	access to $A$ is through queries from $\Lambda$. Suppose that,
	on each input $A$, its set of queried evaluations were exactly
	$\Lambda_\Gamma(A)$. Since both prescribed sets are nonempty,
	the procedure must issue a first query. Before receiving any
	oracle answer, that first query is independent of $A$.
	It would therefore belong to
	\[
	\Lambda_\Gamma(0)\cap\Lambda_\Gamma(1)=
	\{\lambda_0\}\cap\{\lambda_1\}=\varnothing,
	\]
	a contradiction.
	
	The output map $\Gamma$ nevertheless has a sequential realization:
	query $\lambda_0$ and return its value. Its resulting fixed
	query policy differs from the specified policy $\Lambda_\Gamma$.
\end{example}

\begin{remark}[Suppressing the query policy]\label{rem:suppress-query-policy}
	When only the output map matters, we will (ab)use the notation and write $\Gamma$ for the
	pair $(\Gamma,\Lambda_\Gamma)$, keeping $\Lambda_\Gamma$ implicit; when needed we
	write $\Lambda_\Gamma(A)$ explicitly.
	This avoids ambiguity later when towers are written only in terms of their output maps.
\end{remark}

\begin{definition}[Standard extensional tower of algorithms]\label{def:standard-tower}
	Let
	\[
	\mathcal P=(\Xi,\Omega,(\mathcal M,d),\Lambda)
	\]
	be an computational problem and let \(k\in\mathbb N_0\).
	
	A standard tower of height \(0\) for \(\mathcal P\) is a single general algorithm $(\Gamma,\Lambda_\Gamma)$ such that
	\[
	\Gamma(A)=\Xi(A)
	\]
	for $A\in\Omega$. For \(k\in \mathbb{N}\), a standard tower of height \(k\) for \(\mathcal P\) consists of general algorithms
	\[
	\Gamma_{n_k,\ldots,n_1}:\Omega\to \mathcal M, \, (n_1,\ldots,n_k)\in\mathbb N^k,
	\]
	such that for every \(A\in\Omega\) the iterated limit
	\[
	\Xi(A) = \lim_{n_k\to\infty} \lim_{n_{k-1}\to\infty} \cdots \lim_{n_1\to\infty} \Gamma_{n_k,\ldots,n_1}(A)
	\]
	exists in the metric \(d\).
	
	All intermediate limits are required to exist in $\mathcal M$ for every input and every fixed tuple of remaining indices. More explicitly, for $j=1,\ldots,k-1$ define
	\[
	\Gamma_{n_k,\ldots,n_{j+1}}(A) :=\lim_{n_j\to\infty}\Gamma_{n_k,\ldots,n_j}(A),
	\]
	and require
	\[
	\Xi(A)=\lim_{n_k\to\infty}\Gamma_{n_k}(A).
	\]
	When $k=1$, the first family of conditions is empty.
	Importantly: The intermediate maps are just required to be semantic pointwise limits; only the deepest-level maps $\Gamma_{n_k,\ldots,n_1}$ are required to be general algorithms.
\end{definition}

We also call the deepest-level approximants of a tower its \textit{leaf maps}. At height $0$ the sole approximating map is its leaf.

A family of approximants may be specified separately, or it may be produced by one procedure uniformly in the tower indices and the input. These are different requirements. We make the latter precise below for oracle BSS machines and, later, for towers operating on names. The two settings also use different output interfaces: a finite exact code in the BSS setting and an infinite name in the Type-2 setting. We remark here that the distinction between individually computable approximants and a jointly computable indexed family is also made explicitly in \cite[\S~4, p.~9]{Wyeth26}, in the paragraph on sequences of functions.

Uniform machine interpretations already occur explicitly in the SCI literature; see \cite[Def.~A.4]{colbrook2024limits} and \cite[Def.~5.3, Rem.~5.4]{Colbrook24Geo}.
We fix one such interpretation, with an explicit oracle enumeration, a parameter convention, and a finite-output decoder to sharpen the picture.
The uniform point-valued tower convention in \cite[Def.~42]{neumann2018topological} is a related special setting; comparisons with it must retain the specified input, output, and coding conventions.

Write for this in the following 
\[
	\mathbb R^\star:=\bigsqcup_{m\in\mathbb N_0}\mathbb R^m
\]
for the set of finite real tuples, with the length included as part of the code. In particular, $\mathbb R^0$ consists of the empty tuple.

\begin{definition}[Finite output interface]\label{def:finite-output-interface}
	Let $(\mathcal{M},d)$ be a metric space.
	A \textit{finite output interface} for $\mathcal{M}$ is a fixed (possibly partial) map
	\[
	\delta^{\mathrm{fin}}_{\mathcal{M}}:\subseteq \R^\star \to \mathcal{M}.
	\]
	Elements $u\in\dom(\delta^{\mathrm{fin}}_{\mathcal{M}})$ are called \textit{finite codes} for
	$\delta^{\mathrm{fin}}_{\mathcal{M}}(u)\in\mathcal{M}$.
\end{definition}

\begin{definition}[Evaluation oracle for a countable evaluation set]\label{def:eval-oracle}
	Assume $\Lambda$ is countable and fix an enumeration $\Lambda=(f_i)_{i\in\N}$ from \Cref{def:evaluation-table}. For each $A\in\Omega$ define the \textit{evaluation oracle}
	\[
	\mathcal{O}_A:\N\to\R^2,\qquad \mathcal{O}_A(i):=\bigl(\operatorname{Re}(f_i(A)),\operatorname{Im}(f_i(A))\bigr).
	\]
\end{definition}

\begin{definition}[Oracle BSS machine (parameter-free; partial QUERY)]\label{def:oracle-bss}
	Fix an enumeration $\Lambda=(f_i)_{i\in\N}$ and the associated oracle maps
	$\mathcal{O}_A:\N\to\R^2$ from \Cref{def:eval-oracle}.
	An \textit{oracle BSS machine over $(\R,+,-,\cdot,/,<)$} in the sense of \cite{BSS98}
	\textit{without real parameters} is a BSS machine whose only built-in constants are
	$0$ and $1$ (hence all rational constants are computable inside the machine),
	equipped with an additional instruction
	\[
	\mathrm{QUERY}: \
	\text{read } i \text{ from a designated register and write } \mathcal{O}_A(i)\in\R^2
	\text{ into designated registers.}
	\]
	The instruction $\mathrm{QUERY}$ is \textit{partial}: it is defined only if the designated
	register contains a natural number $i\in\N$ (encoded as the real $i$); otherwise the run
	is undefined (equivalently: the machine diverges).
	
	A run \textit{halts} if it reaches a halting state after finitely many steps; upon halting it
	outputs a finite tuple $u\in\R^\star$ (encoded in the standard BSS way). We keep the common convention that if a run includes division by $0$, it is undefined.
\end{definition}

This is the formal replacement for \enquote{finite computation using arithmetic operations and comparisons} as appearing e.g. in \cite[Def.~A.4]{colbrook2024limits}.

\begin{definition}[BSS-arithmetic algorithm]\label{def:BSS-arith-alg}
	Let $\mathcal{P}=\bigl( \Xi,\Omega,(\mathcal{M},d),\Lambda \bigr)$ be a computational problem.
	Assume:
	\begin{enumerate}
		\item $\Lambda$ is countable with a fixed enumeration $\Lambda=(f_i)_{i\in\N}$;
		\item $\mathcal{M}$ is equipped with a fixed finite output interface
		$\delta^{\mathrm{fin}}_{\mathcal{M}}$ as in \Cref{def:finite-output-interface}.
	\end{enumerate}
	A map $\Gamma:\Omega\to\mathcal{M}$ is called a \textit{BSS-arithmetic algorithm}
	if there exists an oracle BSS machine $\mathbb{M}$ in the sense of \Cref{def:oracle-bss}
	such that for every $A\in\Omega$ the run of $\mathbb{M}$ with oracle $\mathcal{O}_A$
	halts and outputs some $u\in\dom(\delta^{\mathrm{fin}}_{\mathcal{M}})$ with
	\[
	\delta^{\mathrm{fin}}_{\mathcal{M}}(u)=\Gamma(A).
	\]
	We require $\mathbb M$ to be parameter-free in the sense of \Cref{def:oracle-bss}.
\end{definition}

\begin{lemma}[Oracle BSS computations induce general algorithms]\label{lem:bss-raw}
	Every BSS-arithmetic algorithm with respect to a nonempty enumerated $\Lambda$ admits a general-algorithm query policy.
\end{lemma}
\begin{proof}
	Fix a witnessing deterministic oracle BSS machine, with any ordinary input parameters held fixed. Let $I(A)\subseteq\mathbb N$ be the finite set of indices queried during its halting run on $A$, and set
	\[
		Q(A):=\{f_i:i\in I(A)\}\cup\{f_1\}.
	\]
	This is finite and nonempty. If $B$ agrees with $A$ on $Q(A)$, induction through the run shows that the computations visit the same states and receive identical answers at every oracle instruction. They therefore halt with the same finite output code, and $I(B)=I(A)$. Applying the fixed decoder gives the same output point, while the equality of index sets gives $Q(B)=Q(A)$.
\end{proof}

\begin{definition}[Tower types: general (G) vs.\ arithmetic (A)]\label{def:tower-types-BSS}
	Let
	\[
	\mathcal P=(\Xi,\Omega,(\mathcal M,d),\Lambda)
	\]
	be an computational problem.
	\begin{itemize}
		\item \textbf{Raw general towers ($\alpha=G$):} A type-G tower is a standard tower in the sense of \Cref{def:standard-tower}.
		\item \textbf{Uniformly BSS-arithmetic towers ($\alpha=A$):} assume $\Lambda$ is countable with fixed enumeration and
		$(\mathcal{M},d)$ comes with a fixed finite output interface as in
		\Cref{def:BSS-arith-alg}. A height-\(k\) BSS-arithmetic tower is a standard tower whose deepest-level approximants
		\[
		\Gamma_{n_k,\ldots,n_1}
		\]
		are produced uniformly by one parameter-free oracle BSS machine. More precisely, there is a single oracle BSS machine \(\mathbb M\) such that, for every
		\[
		(n_k,\ldots,n_1)\in\mathbb N^k
		\quad\text{and every }A\in\Omega,
		\]
		the machine \(\mathbb M\), given the indices \((n_k,\ldots,n_1)\) in its input
		registers and oracle access to \(O_A\), halts and outputs a code
		\[
		u_{n_k,\ldots,n_1}(A)\in\operatorname{dom}(\delta_\mathcal M^{\rm fin})
		\]
		with
		\[
		\delta_\mathcal M^{\rm fin}(u_{n_k,\ldots,n_1}(A))
		=
		\Gamma_{n_k,\ldots,n_1}(A).
		\]
		For height \(0\), the same definition is read with no index input.
	\end{itemize}
\end{definition}

\begin{remark}[Implementation data]
	The enumeration $(f_i)$ and the finite output interface $\delta_{\mathcal M}^{\mathrm{fin}}$ are part of the fixed implementation data. All BSS-arithmetic heights are relative to these choices.
	Keep in mind that \textit{parameter-free} constrains the machine program, but does not by itself impose regularity on the decoded output map. For example, for any $B\subseteq\mathbb R$ the decoder $\delta^{\mathrm{fin}}((x))=\mathbf 1_B(x)$ allows a machine with the single evaluation $\operatorname{id}_{\mathbb R}$ to compute $\mathbf 1_B$ by returning its oracle answer as a one-real code.
	Further decoder assumptions must therefore be stated whenever regularity or effectivity of decoded outputs is used.
\end{remark}

\begin{definition}[SCI]\label{def:sci}
	Let \(\mathcal P\) be an computational problem.
	
	\begin{enumerate}
		\item[(S1)] The raw type-G SCI is defined by
		\[
		\SCIG(\mathcal P):=
		\min\{k\in\mathbb N_0:
		\text{there exists a standard raw type-G tower of height }k\text{ for } \mathcal P\},
		\]
		with value \(\infty\) if the set is empty.
		
		\item[(S2)] When the BSS-arithmetic implementation of \Cref{def:tower-types-BSS} is available, define
		\[
		\SCIA(\mathcal P) := \min\{k\in\mathbb N_0:
		\text{there exists a uniformly BSS-arithmetic tower of height }k\text{ for } \mathcal P\},
		\]
		with value \(\infty\) if the set is empty.
	\end{enumerate}
\end{definition}

To illustrate these notations we provide a short example.

\begin{example}[Rationality testing: raw and arithmetic heights]\label{ex:rationality-early}
	Let $\Omega=\mathbb R$, $\Lambda=\{\operatorname{id}_{\mathbb R}\}$,
	$\mathcal M=2$ with the discrete metric, and $\Xi=\mathbf 1_{\mathbb Q}$. For the BSS interpretation, enumerate $\Lambda$ by $f_i=\operatorname{id}_{\mathbb R}$ for every $i\in\mathbb N$, and fix the finite output interface
	\[
		\operatorname{dom}(\delta_{\mathcal M}^{\mathrm{fin}})=\{(0),(1)\}\subseteq\mathbb R^\star, \quad \delta^{\rm fin}_{\mathcal M}((b))=b
	\]
	for $b \in 2$. Then
	\[
	\SCIG(\mathcal P)=0, \quad \SCIA(\mathcal P)=1.
	\]
	The raw height is $0$ because one exact identity query determines the input and arbitrary post-processing is allowed. We will also \textit{sketch} a proof for the arithmetical case:
	
	For the upper bound, let
	\[
	F_n=\{a/b:a\in\mathbb Z,\ |a|\leq n,\ b\in\mathbb N,\ b\leq n\}.
	\]
	One parameter-free BSS machine, given $n$, tests the finitely many equalities $bx=a$ and returns $\Gamma_n(x)=\mathbf 1_{F_n}(x)$. Since $F_n\subseteq F_{n+1}$ and $\bigcup_nF_n=\mathbb Q$, $\Gamma_n(x)\to\Xi(x)$ for every $x$.
	
	For the lower bound, suppose a parameter-free BSS machine computes $\Xi$ without a limit. Run it at a transcendental real $t$. Along this finite halting computation, every register value is a rational function of $t$ with rational coefficients. A tested rational function which vanishes at $t$ vanishes identically; every nonzero test and every denominator keeps its sign on a sufficiently small neighbourhood of $t$. Moreover, a queried address $r(t)\in\mathbb N$
	forces $r$ to be that constant integer, so these queries remain valid
	on the same neighbourhood. Thus the finite computation path persists.
	Its output code has length $1$ and value $0$; the corresponding
	rational function is identically $0$. The machine therefore also
	returns $0$ at a nearby rational input, contradiction.
\end{example}

At this point one might ask if computational problems $\mathcal P$ with $\SCIG(\mathcal P)=\infty$ exist, which would ensure the existence of \textit{unsolvable} computational problems for the SCI. Based on an alternating quantifier decision problem for matrices analyzed in \cite[Thm.~5.17, Thm.~5.19]{Colbrook24Geo} we will construct such a problem.

\begin{definition}[Alternating predicates on a common matrix interface]\label{def:alternating-matrices}
Let
\[
 \Omega_{\mathrm{mat}}:= 2^{\mathbb N^2},\quad
 \Lambda_{\mathrm{mat}}:= \{\operatorname{ev}_{a,b}:(a,b)\in\mathbb N^2\},\quad \operatorname{ev}_{a,b}(A):=A(a,b).
\]
We use the product topology on $\Omega_{\mathrm{mat}}$ and fix a computable enumeration of $\mathbb N^2$. Fix bijections $\beta_m:\mathbb N^m\to\mathbb N^2$ whose evaluations and inverses are computable uniformly in $m\in\mathbb N$, with $\beta_2(n_1,n_2)=(n_1,n_2)$.

Let $Q_i=\exists$ for odd $i$ and $Q_i=\forall$ for even $i$. Define $\Xi_m:\Omega_{\mathrm{mat}}\to 2$ by
\[
 \Xi_m(A)=1 \iff (Q_1n_1)\cdots(Q_mn_m) \bigl( A(\beta_m(n_1,\ldots,n_m))=1 \bigr).
\]
This defines the computational problem
\[
 \mathcal P_m^{\mathrm{mat}}:=(\Xi_m,\Omega_{\mathrm{mat}},(2,d_{\mathrm{d}}),\Lambda_{\mathrm{mat}}).
\]
\end{definition}

For the following calibration, $\Sigma^0_1(X)$ denotes the open sets of a metrizable space $X$,
$\Pi^0_r(X)$ consists of complements of $\Sigma^0_r(X)$ sets, and $\Sigma^0_{r+1}(X)$ consists of countable unions of $\Pi^0_r(X)$ sets. Write further $\Delta^0_r(X):=\Sigma^0_r(X)\cap\Pi^0_r(X)$. This defines the well known Borel hierarchy, which is strict on $\Omega_{\mathrm{mat}}$, see e.g. \cite[Thm.~22.4]{Kec95}.

\begin{lemma}[Calibration of the alternating predicates]\label{lem:alternating-calibration}
	For every $m\in\mathbb N$,
	\[
		\SCIG(\mathcal P_{\mathrm{mat}}^m)=m.
	\]
\end{lemma}

\begin{proof}
	Let $S_m:=\Xi_m^{-1}(\{1\})\subseteq\Omega_{\mathrm{mat}}$. The defining alternating quantifiers show that
	$S_m\in\Sigma^0_m(\Omega_{\mathrm{mat}})$. We first verify completeness under continuous reductions from $2^{\mathbb N_0}$. Every $E\in\Sigma^0_m(2^{\mathbb N_0})$ can be written as
	\[
	x\in E \iff
	(\exists n_1)(\forall n_2)\cdots(Q_m n_m) [x\in C_{n_1,\ldots,n_m}],
	\]
	where every $C_{n_1,\ldots,n_m}$ is clopen.
	For $m=1$, this follows from the countable clopen basis of $2^{\mathbb{N}_0}$. The general assertion, together with its dual starting with a universal quantifier, follows by induction using countable unions and complements in the definition of the Borel hierarchy.
	
	Define $A_x\in\Omega_{\mathrm{mat}}$ by
	\[
	A_x\bigl(\beta_m(n_1,\ldots,n_m)\bigr) :=\mathbf1_{C_{n_1,\ldots,n_m}}(x).
	\]
	The bijectivity of $\beta_m$ makes this a well-defined matrix. Each coordinate depends continuously on $x$, so $x\mapsto A_x$ is continuous and $x\in E$ iff $A_x\in S_m$. Thus $S_m$ is $\Sigma^0_m$-complete. If it belonged to $\Pi^0_m$, closure under continuous inverse images would force every $\Sigma^0_m$ in $2^{\mathbb{N}_0}$ to belong to $\Pi^0_m$, in contradiction to \cite[Thm.~22.4]{Kec95}. Hence
	\[
	S_m\notin\Pi^0_m(\Omega_{\mathrm{mat}}).
	\]
	
	For the lower bound, a raw height $0$ map $h:\Omega_{\mathrm{mat}}\to 2$ has clopen fibres: a finite assignment of answers to coordinate queries determines an open cylinder on which $h$ is constant. For a pointwise convergent sequence of binary-valued maps,
	\[
		\{A:\lim_n h_n(A)=b\} =\bigcup_{N=1}^\infty\bigcap_{n\geq N}\{A:h_n(A)=b\}.
	\]
	for $b \in 2$. It follows inductively that both fibres of a height-$k$ raw tower belong to $\Sigma^0_{k+1}$ and hence to $\Delta^0_{k+1}$, since the two fibres are complements of one another. If $k<m$, this would imply $S_m\in\Pi^0_m$, which is clearly impossible.
	
	For the upper bound, let $s_1,\ldots,s_m\in\mathbb N$ and define $G_{s_1,\ldots,s_m}:\Omega_{\mathrm{mat}}\to 2$ by
	\[
		G_{s_1,\ldots,s_m}(A)=1 \iff (Q_1 n_1\in\{1,\ldots,s_1\})\cdots (Q_m n_m\in\{1,\ldots,s_m\}) [A(\beta_m(n_1,\ldots,n_m))=1].
	\]
	Each bounded existential quantifier is a finite maximum, and each bounded universal quantifier is a finite minimum. The query policy consists then of the fixed finite nonempty set
	\[
		\left\{\operatorname{ev}_{a,b}: (a,b)\in\beta_m\!\left(\prod_{i=1}^m\{1,\ldots,s_i\}\right) \right\}.
	\]
	Every innermost binary limit stabilizes, and finite Boolean operations commute with these limits. Iterating this observation gives all intermediate limits and
	\[
		\Xi_m(A)=\lim_{s_1\to\infty}\cdots \lim_{s_m\to\infty}G_{s_1,\ldots,s_m}(A).
	\]
	After relabelling $n_{m-i+1}=s_i$, this is a height-$m$ standard raw tower.
\end{proof}

We modify this problem now by
\[
\Omega_\infty:=\{(m,A) : m\in\mathbb N,\ A\in\Omega_{\mathrm{mat}}\}.
\]
Equip \(\Omega_\infty\) with the oracle consisting of the tag map
\[
\tau(m,A):=m
\]
together with all matrix-coordinate evaluations
\[
\operatorname{ev}_{a,b}^{\infty}:\Omega_\infty\to2, \qquad
\operatorname{ev}_{a,b}^{\infty}(m,A):=A(a,b).
\]

Set
\[
\Lambda_\infty := \{\tau\}\cup \{\operatorname{ev}_{a,b}^{\infty} : (a,b)\in\mathbb N^2\}.
\]
Define
\[
\Xi_\infty(m,A):=\Xi_m(A).
\]
Let
\[
\mathcal P_\infty^{\mathrm{mat}} := (\Xi_\infty,\Omega_\infty,(2,d_{\mathrm{disc}}),\Lambda_\infty)
\]
be the corresponding computational problem.

Then we get
\begin{theorem}[Existence of impossibility result for general algorithms]\label{thm:Xi-infty-height}
	One has
	\[
	\SCIG(\mathcal P_\infty^{\mathrm{mat}})=\infty.
	\]
\end{theorem}
\begin{proof}
	Suppose that a height-$k$ raw type-$G$ tower computes $\mathcal P_\infty^{\mathrm{mat}}$, where $k\in\mathbb N_0$.
	Put $m=k+1$ and let $\iota_m(A)=(m,A)$. For every deepest-level map $\Gamma$ of the tower, define
	\[
	\widetilde\Gamma(A)=\Gamma(\iota_m(A)),\quad \widetilde Q(A)=\{\operatorname{ev}_{a,b}:
	\operatorname{ev}_{a,b}^{\infty}\in \Lambda_\Gamma(\iota_m(A))\} \cup\{\operatorname{ev}_{1,1}\}.
	\]
	This policy is clearly finite and nonempty. If $B$ agrees with $A$ on $\widetilde Q(A)$, then $\iota_m(B)$ agrees with $\iota_m(A)$ on every query in $\Lambda_\Gamma(\iota_m(A))$: all coordinate answers agree, and the tag has the same constant value $m$.
	\Cref{def:gen-alg} therefore gives both
	\[
	\widetilde\Gamma(B)=\widetilde\Gamma(A),\quad \Lambda_\Gamma(\iota_m(B))=\Lambda_\Gamma(\iota_m(A)).
	\]
	The second equality also gives $\widetilde Q(B)=\widetilde Q(A)$. Thus every restricted deepest map is a general algorithm for $\Lambda_{\mathrm{mat}}$. Restricting the intermediate limits yields a height-$k$ tower for $\Xi_m$, contradicting \Cref{lem:alternating-calibration}.
\end{proof}

One might now think to impose finite-query dependence at every non-final level. This leads to the following concept.

\begin{definition}[Finite information at all nonfinal levels]\label{def:all-level-tower}
	An all-level finite-information tower is a standard raw tower with
	the following additional requirement. If its height is $k\geq2$,
	every intermediate map
	\[
	\Gamma_{n_k,\ldots,n_{j+1}},\qquad j=1,\ldots,k-1,
	\]
	is a general algorithm. For heights zero and one there is no
	additional condition. The final target $\Xi$ is not required to be
	a general algorithm.
	Let $\operatorname{SCI}^{\mathrm{all}}_G(\mathcal P)$ be the least
	height of such a tower, with value $\infty$ if none exists.
\end{definition}

However we show now why this is not a good idea.

\begin{proposition}[Collapse of finite all-level heights]\label{prop:all-level-collapse}
	For every computational problem $\mathcal P$,
	\[
	0<\operatorname{SCI}^{\mathrm{all}}_G(\mathcal P)<\infty
	\quad\Longrightarrow\quad
	\operatorname{SCI}^{\mathrm{all}}_G(\mathcal P)=1.
	\]
\end{proposition}
\begin{proof}
	Suppose \(\mathcal P\) has an all-level tower of height \(k\in \mathbb{N}\). By definition, the outermost maps
	\[
	\Gamma_{n_k}:\Omega\to \mathcal M
	\]
	are themselves general algorithms, and
	\[
	\Xi(A)=\lim_{n_k\to\infty}\Gamma_{n_k}(A)
	\]
	for $A\in\Omega$. Thus \((\Gamma_{n_k})_{n_k\in\mathbb N}\) is already a height-\(1\) type-G tower. Hence
	\[
	\mathrm{SCI}^{\rm all}_G( \mathcal P)\le1.
	\]
	If the height is positive, it is not \(0\), so it must be \(1\).
\end{proof}

\section{Algebraic properties of general algorithms}\label{sec:AlgGA}
As the concept general algorithm carries the word algorithm in it, it might be interesting to check if well known properties/concepts such as closure under composition, currying and the existence of universal algorithms (meaning universal with finite program codes) hold.

For that, suppose $\Omega$ is a set, $\mathcal M$ a metric space, and
$\varnothing\ne\Lambda\subseteq\mathbb C^\Omega$, write
\[
	\mathsf G_\Lambda(\Omega,\mathcal M):= \{\Gamma:\Omega\to\mathcal M:
	\text{there exists a query policy satisfying both clauses of \Cref{def:gen-alg}}\}.
\]

\addtocontents{toc}{\protect\setcounter{tocdepth}{1}}
\subsectionnotoc{Composition}
\addtocontents{toc}{\protect\setcounter{tocdepth}{2}}
 
\begin{proposition}[Arbitrary postcomposition]\label{prop:ArbPostComp}
Let $\mathcal M, \mathcal N$ be metric spaces. If $\Gamma \in \mathsf{G}_{\Lambda}(\Omega,\mathcal M)$, then
\[
	h \circ \Gamma \in \mathsf{G}_{\Lambda}(\Omega,\mathcal N)
\]
for $h : \mathcal M \to \mathcal N$.
\end{proposition}
\begin{proof}
	Using the query policy of $\Gamma$, agreement on its queried evaluations forces the same value of $\Gamma$, hence the policy remains stable.
\end{proof}

A much more interesting algebraic property might be the following.
\begin{proposition}[Query-simulation stability]\label{prop:Querysimstab}
	Let $X,Y$ be sets, $\varnothing\ne\Lambda_X\subseteq\mathbb C^X$ and
	$\varnothing\ne\Lambda_Y\subseteq\mathbb C^Y$, and let $F:X\to Y$.
	For each $\lambda\in\Lambda_Y$, suppose that $\lambda\circ F\in \mathsf G_{\Lambda_X}(X,\mathbb C)$, and fix a witnessing query policy $Q_\lambda$.
	If $\Gamma\in\mathsf G_{\Lambda_Y}(Y,\mathcal M)$, where $\mathcal M$ is a metric
	space, then $\Gamma\circ F\in\mathsf G_{\Lambda_X}(X,\mathcal M)$.
\end{proposition}
\begin{proof}
	Fix a policy $\Lambda_\Gamma$ for $\Gamma$ and put
	\[
		Q(x)=\bigcup_{\lambda\in \Lambda_\Gamma(F(x))}Q_\lambda(x).
	\]
	This is finite and nonempty. Suppose that $x'$ agrees with $x$
	on $Q(x)$. For every $\lambda\in \Lambda_\Gamma(F(x))$, the policy
	$Q_\lambda$ gives
	\[
	\lambda(F(x'))=\lambda(F(x)),
	\qquad Q_\lambda(x')=Q_\lambda(x).
	\]
	By \Cref{def:gen-alg}
	\[
	\Gamma(F(x'))=\Gamma(F(x)),\quad \Lambda_\Gamma(F(x'))=\Lambda_\Gamma(F(x)).
	\]
	Using both policy equalities in the displayed union yields $Q(x')=Q(x)$.
\end{proof}

\addtocontents{toc}{\protect\setcounter{tocdepth}{1}}
\subsectionnotoc{Currying}
\addtocontents{toc}{\protect\setcounter{tocdepth}{2}}

\textbf{Identities and interface morphisms.}$\,$ The point-output classes $\mathsf G_{\Lambda_X}(X,Y)$ need not contain identity maps. For example, equip $\mathcal C$ with coordinate evaluations
\[
	\Lambda_{\mathcal C}:=\{e_j:j\in\mathbb N_0\},\quad e_j(x):=x(j).
\]
Then
\[
	\operatorname{id}_{\mathcal C}\notin \mathsf G_{\Lambda_{\mathcal C}}(\mathcal C, \mathcal C).
\]
Indeed, a finite query set omits some coordinate $j$. Changing only that coordinate preserves every queried value but changes the entire output point.

Conversely, whenever $X$ is a metric space and $\operatorname{id}_X\in \mathsf G_{\Lambda_X}(X,X)$, \Cref{prop:ArbPostComp} gives
\[
	\mathsf G_{\Lambda_X}(X,\mathcal M)=\mathcal M^X
\]
for every metric space $\mathcal M$.

These observations concern finite determination of an entire output point. A compositional notion of map between evaluation interfaces can instead be defined by requiring finite-query simulation of each target evaluation.

\begin{definition}[Interface morphisms]\label{def:interface-morphism}
	An evaluation interface is a pair $(X,\Lambda_X)$, where $X$ is a set and $\varnothing\ne\Lambda_X\subseteq\mathbb C^X$.
	A map
	\[
		F:(X,\Lambda_X)\longrightarrow(Y,\Lambda_Y)
	\]
	is an interface morphism if
	\[
		\lambda\circ F\in \mathsf G_{\Lambda_X}(X,\mathbb C)
	\]
	for every $\lambda\in\Lambda_Y$.
\end{definition}

\begin{proposition}[Composition and products]\label{prop:interface-category}
	Evaluation interfaces and interface morphisms form a category. For interfaces $(X,\Lambda_X)$ and $(Y,\Lambda_Y)$, their categorical product is
	\[
		(X\times Y,\Lambda_{X\times Y}),
	\]
	where, writing $\pi_X$ and $\pi_Y$ for the projections,
	\[
		\Lambda_{X\times Y}:=\{\lambda\circ\pi_X:\lambda\in\Lambda_X\} \cup \{\mu\circ\pi_Y:\mu\in\Lambda_Y\}.
	\]
\end{proposition}

\begin{proof}
	For every $\lambda\in\Lambda_X$, the map $\lambda\circ\operatorname{id}_X=\lambda$ is a general algorithm with fixed policy $\{\lambda\}$. Hence identities are morphisms.
	
	Let $F:(X,\Lambda_X)\to(Y,\Lambda_Y)$ and $G:(Y,\Lambda_Y)\to(Z,\Lambda_Z)$ be morphisms. For $\nu\in\Lambda_Z$, the map $\nu\circ G$ belongs to $\mathsf G_{\Lambda_Y}(Y,\mathbb C)$.
	\Cref{prop:Querysimstab} therefore gives
	\[
		\nu\circ G\circ F\in \mathsf G_{\Lambda_X}(X,\mathbb C).
	\]
	Thus $G\circ F$ is a morphism. Associativity and the identity laws are inherited from ordinary function composition.
	
	Each product projection is a morphism because its pullback of a target evaluation belongs to $\Lambda_{X\times Y}$. For maps $F:Z\to X$ and $G:Z\to Y$, the pairing
	\[
		\langle F,G\rangle:Z\to X\times Y
	\]
	is a morphism precisely when $F$ and $G$ are morphisms: the pullbacks of the product evaluations are exactly $\lambda\circ F$ and $\mu\circ G$. The usual uniqueness of pairing proves the product universal property.
\end{proof}

\textbf{Currying of point-output general algorithms.}$\,$ We now return to the original classes $\mathsf G_\Lambda$. For an interface $(Y,\Lambda_Y)$, let $\tau_{\Lambda_Y}^{=}$
be the topology generated by
\[
	\{\mu^{-1}(\{z\}):\mu\in\Lambda_Y,\ z\in\mathbb C\}.
\]
Finite intersections of these sets are called exact-observation cylinders. This topology is the natural topology on the interface sets as we will see in more detail in \Cref{sec:ReprSpLimHier}.

\begin{proposition}[Currying with a compact parameter space]\label{prop:compact-currying}
	Let $X,Y$ be nonempty sets with nonempty evaluation interfaces $\Lambda_X,\Lambda_Y$, and use the product interface $\Lambda_{X\times Y}$ from \Cref{prop:interface-category}.
	Let $(\mathcal M,d)$ be a nonempty metric space and put
	\[
		\mathcal F:= \mathsf G_{\Lambda_Y}(Y,\mathcal M).
	\]
	Equip $\mathcal F$ with the bounded uniform metric
	\[
		d_{\mathcal F}(h,g):= \sup_{y\in Y}\min\{1,d(h(y),g(y))\}.
	\]
	
	If $(Y,\tau_{\Lambda_Y}^{=})$ is compact, ordinary currying restricts to a bijection
	\[
	\begin{split}
		\mathsf G_{\Lambda_{X\times Y}}(X\times Y,\mathcal M)
		&\rightarrow
		\mathsf G_{\Lambda_X}(X,\mathcal F),\\
		H&\mapsto \operatorname{cur}(H),
	\end{split}
	\]
	where
	\[
		\operatorname{cur}(H)(x)(y):=H(x,y).
	\]
	Its inverse (called uncurrying) is
	\[
		\operatorname{uncur}(F)(x,y):=F(x)(y).
	\]
	The uncurrying direction holds without compactness.
\end{proposition}

\begin{proof}
	Suppose first that
	\[
		H\in G_{\Lambda_{X\times Y}}(X\times Y,\mathcal M),
	\]
	and fix a witnessing query policy $Q_H$. Define
	\[
	\begin{split}
		Q_H^X(x,y)
		&:=
		\{\lambda\in\Lambda_X: \lambda\circ\pi_X\in Q_H(x,y)\},\\
		Q_H^Y(x,y)
		&:=
		\{\mu\in\Lambda_Y: \mu\circ\pi_Y\in Q_H(x,y)\}.
	\end{split}
	\]
	Both sets are finite: the two pullback maps are injective because $X$ and $Y$ are nonempty.
	
	Fix $\mu_*\in\Lambda_Y$. For every $x\in X$, the section
	\[
		H_x:Y\to \mathcal M,\qquad H_x(y):=H(x,y),
	\]
	is a general algorithm with policy
	\[
		Q_x(y):=Q_H^Y(x,y)\cup\{\mu_*\}.
	\]
	Indeed, if $y'$ agrees with $y$ on $Q_x(y)$, then $(x,y')$ agrees with $(x,y)$ on every query in $Q_H(x,y)$. Consequently
	\[
		H(x,y')=H(x,y),\quad Q_H(x,y')=Q_H(x,y).
	\]
	The latter equality also gives $Q_x(y')=Q_x(y)$. Thus $H_x\in\mathcal F$.
	
	For fixed $x\in X$ and $y\in Y$, put
	\[
		V_{x,y}:=\{z\in Y: (\forall\mu\in Q_H^Y(x,y))\ \mu(z)=\mu(y)\}.
	\]
	This is an open neighbourhood of $y$. For every $z\in V_{x,y}$, stability of $Q_H$ gives
	\[
		Q_H(x,z)=Q_H(x,y).
	\]
	Compactness supplies $y_1,\ldots,y_r\in Y$ such that
	\[
		Y=\bigcup_{i=1}^r V_{x,y_i}.
	\]
	Hence
	\[
		\{Q_H(x,y):y\in Y\}\subseteq \{Q_H(x,y_1),\ldots,Q_H(x,y_r)\},
	\]
	so the policy has only finitely many values when $x$ is fixed.
	
	Fix $\lambda_*\in\Lambda_X$ and define
	\[
		Q_{\mathrm{cur}}(x):=\{\lambda_*\} \cup \bigcup_{y\in Y}Q_H^X(x,y).
	\]
	The preceding finite-image conclusion shows that this set is finite; it is nonempty by construction.
	
	Suppose that $x'$ agrees with $x$ on $Q_{\mathrm{cur}}(x)$. For every $y\in Y$, the pairs $(x',y)$ and $(x,y)$ agree on every query in $Q_H(x,y)$. Therefore
	\[
		H(x',y)=H(x,y),\quad Q_H(x',y)=Q_H(x,y) \quad(y\in Y).
	\]
	The first equality gives $H_{x'}=H_x$. The second gives
	\[
		Q_H^X(x',y)=Q_H^X(x,y) \quad(y\in Y),
	\]
	and thus
	\[
		Q_{\mathrm{cur}}(x')=Q_{\mathrm{cur}}(x).
	\]
	Hence $Q_{\mathrm{cur}}$ witnesses
	\[
		\operatorname{cur}(H)\in G_{\Lambda_X}(X,\mathcal F).
	\]
	
	Conversely, let
	\[
		F\in G_{\Lambda_X}(X,\mathcal F)
	\]
	with witnessing policy $Q_F$. For each $h\in\mathcal F$, choose one witnessing policy $Q_h$. Define
	\[
		H(x,y):=F(x)(y)
	\]
	and
	\[
		Q(x,y):=\{\lambda\circ\pi_X:\lambda\in Q_F(x)\} \cup
		\{\mu\circ\pi_Y:\mu\in Q_{F(x)}(y)\}.
	\]
	This policy is finite and nonempty.
	
	Suppose that $(x',y')$ agrees with $(x,y)$ on $Q(x,y)$. The $X$-queries imply
	\[
		F(x')=F(x),\quad Q_F(x')=Q_F(x).
	\]
	Write $h:=F(x')=F(x)$. The $Y$-queries then imply
	\[
		h(y')=h(y),\quad Q_h(y')=Q_h(y).
	\]
	Consequently
	\[
		H(x',y')=H(x,y), \quad Q(x',y')=Q(x,y).
	\]
	Thus $H\in G_{\Lambda_{X\times Y}}(X\times Y,\mathcal M)$.
\end{proof}

\begin{example}[Failure without compactness]\label{ex:currying-without-compactness}
	Let
	\[
		X:=\mathcal C:=2^{\mathbb N_0},\quad
		\Lambda_{X}=\{e_j:j\in\mathbb N_0\},\quad e_j(x)=x(j),
	\]
	and put
	\[
		Y=\mathbb N_0,\quad
		\Lambda_Y=\{\operatorname{id}_{\mathbb N_0}\},\quad
		\mathcal M=2
	\]
	with the discrete metric. The exact-observation topology on $Y$ is discrete and noncompact. Every map $Y\to 2$ is a general algorithm, since the identity query determines its input. Thus
	\[
		\mathcal F=X,
	\]
	and $d_{\mathcal F}$ is the discrete metric.
	
	Define
	\[
		H:X\times Y\to2,\qquad H(x,n):=x(n).
	\]
	A witnessing policy is
	\[
		Q_H(x,n):=\{\operatorname{id}_{\mathbb N_0}\circ\pi_Y,\, e_n\circ\pi_X\}.
	\]
	Agreement on the first query fixes $n$, and agreement on the second fixes $H(x,n)$; the policy is therefore stable.
	
	However,
	\[
		\operatorname{cur}(H)(x)=x.
	\]
	Changing an unqueried bit preserves every answer in any finite $X$-query set but changes this output. Hence
	\[
		\operatorname{cur}(H)\notin \mathsf G_{\Lambda_X}(X,\mathcal F).
	\]
\end{example}

\begin{remark}[No countable universal family of raw maps]
For $\Omega=\mathbb N$, $\Lambda=\{\operatorname{id}_{\mathbb N}\}$ and $\mathcal M=2$, every map $\mathbb N\to 2$ is a general algorithm. There are $2^{\aleph_0}$ such maps, so no family indexed by finite program codes contains them all, even if its members may be partial.
\end{remark}

\section{Direct but important logical consequences of the Type-G setting}

\begin{definition}[Finite-query factorisation]\label{def:finite-factor}
	Let $\mathcal{P}=(\Xi,\Omega,(\mathcal{M},d),\Lambda)$ be a computational problem. We say that $\Xi$ admits a \textit{finite-query factorisation through $\Lambda$} if there exist $m\in \mathbb N$, $f_1,\dots,f_m\in\Lambda$ and a map $G: \operatorname{im}(f_1,\dots,f_m)\to \mathcal{M}$ such that
	\[
		\Xi(A)=G\bigl(f_1(A),\dots,f_m(A)\bigr)\quad\forall A\in\Omega,
	\]
	where $\operatorname{im}(f_1,\dots,f_m)\subseteq \C^m$ is the range of $A\mapsto (f_1(A),\dots,f_m(A))$.
\end{definition}

\begin{theorem}[Finite-query factorisations force $\mathrm{SCI}_G=0$]\label{thm:finite-factor-scig0}
	Let $\mathcal{P}=( \Xi,\Omega,(\mathcal{M},d),\Lambda)$ be a computational problem. Suppose $\Xi$ admits a finite-query factorisation through $\Lambda$.
	Then $\SCIG(\mathcal{P})=0$.
\end{theorem}
\begin{proof}
	Define $\Gamma:\Omega\to \mathcal{M}$ by $\Gamma(A):=G(f_1(A),\dots,f_m(A))$ and define the query policy
	$\Lambda_\Gamma(A):=\{f_1,\dots,f_m\}$ for all $A\in\Omega$.
	
	Then $(\Gamma,\Lambda_\Gamma)$ is a general algorithm: if $B$ agrees with $A$ on all queries in
	$\Lambda_\Gamma(A)$, then $f_i(B)=f_i(A)$ for all $i$, hence $\Gamma(B)=\Gamma(A)$, and clearly also
	$\Lambda_\Gamma(B)=\Lambda_\Gamma(A)$ because the query set is constant.
	
	Moreover $\Gamma=\Xi$ by construction, so $\Xi$ is computed by a height-$0$ tower (i.e.\ a single
	general algorithm). Therefore $\mathrm{SCI}_G(\mathcal{P})=0$.
\end{proof}

\begin{remark}[Core logical point]\label{rem:core-logical}
	\Cref{thm:finite-factor-scig0} is \textit{pure logic}: once base post-processing is unrestricted,
	\textit{any} map $G$ on a tuple of answers to finitely many fixed evaluations is allowed.
	This is the reason why descriptive-set/Type-2 lower-bound methods do not automatically imply type-G lower bounds.
\end{remark}

\begin{theorem}[Monotonicity in tower type]\label{thm:monotone-G-vs-BSS-A}
	For every computational problem $\mathcal P$ for which $\SCIA(\mathcal P)$
	is defined, we have
	\[
		\mathrm{SCI}_{G}(\mathcal P)\ \le\ \SCIA(\mathcal P).
	\]
\end{theorem}
\begin{proof}
	If $\SCIA(\mathcal P)=\infty$, the
	claimed inequality is immediate. Otherwise choose a witnessing BSS-arithmetic tower of finite height $k$.
	Every deepest-level map is a general algorithm by \Cref{lem:bss-raw}, and the limits are unchanged. The same tower is therefore a raw type-$G$ tower of height $k$, which proves the claim.
\end{proof}

In the following we give it a first try to compare type-G SCI with TTE directly (meaning, without further premisses, so in the most general setting) and see that this is impossible. But for this we have to repeat a few basic notations and results from Weihrauch theory.

\subsection{Represented spaces and the limit hierarchy}\label{sec:ReprSpLimHier}

This Weihrauch recap follows closely the definitions and conventions of \cite[Defs.~11.2.1-11.2.4 and 11.3.1]{brattka2021weihrauch}. Now fix a computable bijection
\[
	\langle\cdot,\cdot\rangle:\mathbb N_0^2\longrightarrow\mathbb N_0.
\]
For streams $p,q \in \Baire$ we also use the computable pairing
\[
\langle p,q\rangle(2j)=p(j), \quad
\langle p,q\rangle(2j+1)=q(j) \quad(j\in\mathbb N_0),
\]
with its computable projections.

\begin{definition}[Represented spaces, problems, and realizers]\label{def:represented-spaces}
	A represented space is a pair
	\[
		\mathbf X=(X,\delta_X), \quad
	\delta_X:\subseteq\mathcal B\longrightarrow X,
	\]
	where \(X\) is a set and \(\delta_X\) is a partial surjection called representation in the following.
	An element \(p\in\operatorname{dom}(\delta_X)\) with \(\delta_X(p)=x\) is called a name of \(x\). A point is computable if it has a computable name.
	
	Let likewise \(\mathbf Y=(Y,\delta_Y)\) be a represented space. A problem
	\[
		f:\subseteq\mathbf X\rightrightarrows\mathbf Y
	\]
	assigns a nonempty set \(f(x)\subseteq Y\) to each \(x\in\operatorname{dom}(f)\subseteq X\).
	For single-valued maps we identify a value with its singleton. Put
	\[
	D_f := \{p\in\operatorname{dom}(\delta_X) \mid \delta_X(p)\in\operatorname{dom}(f)\}.
	\]
	A partial function \(F:\subseteq\mathcal B\to\mathcal B\) realizes \(f\), written \(F\vdash f\), if for every \(p\in D_f\),
	\[
		p\in\operatorname{dom}(F),\quad F(p)\in\operatorname{dom}(\delta_Y),\quad \delta_Y(F(p))\in f(\delta_X(p)).
	\]
	The problem \(f\) is Type-2 computable if it has a partial computable realizer. Computability on \(\mathcal B\) is understood in the usual oracle Turing-machine sense.
\end{definition}

Furthermore, we equip $\mathbb N_0$ in the following with its discrete representation $\delta_{\mathbb N_0}(p)=p(0)$ and use the product representations for its finite powers; we represent \(\mathcal B\) then by the identity too. Subsets of represented spaces carry the corresponding restricted representations unless specified otherwise. The product of \(\mathbf X=(X,\delta_X)\) and \(\mathbf Y=(Y,\delta_Y)\) has the representation
\[
	\delta_{\mathbf X\times\mathbf Y}(\langle p,q\rangle) =(\delta_X(p),\delta_Y(q)),
\]
defined precisely when both coordinates are valid names. For a sequence \((\mathbf X_i)_{i\in\mathbb N}\), the countable product representation reads its \(i\)-th name from the stream
\[
	p_i(j):=p(\langle i-1,j\rangle) \quad (i\in\mathbb N,\ j\in\mathbb N_0).
\]
Thus it sends \(p\) to \((\delta_{X_i}(p_i))_{i\in\mathbb N}\) when every \(p_i\) is a valid name.
These are the standard product constructions, see e.g. \cite[Def.~11.2.5]{brattka2021weihrauch}.
For \(\mathbb R\) and \(\mathbb C\) we use their standard computable metric structures and Cauchy representations. More generally, Cauchy representations of computable metric spaces
are understood as in \cite[Def.~11.2.6]{brattka2021weihrauch}.

The discrete representation of \(2\) is \(\delta_2(p)=p(0)\), with domain \(\{p\in\mathcal B:p(0)\in 2\}\). It is computably equivalent to the Cauchy representation of the discrete two-point metric space.

Two representations of the same set are called computably equivalent if the identity map is Type-2 computable in both directions. The topology induced by a representation \(\delta_X\) is
\[
	\tau_{\delta_X} := \{U\subseteq X: \delta_X^{-1}(U)\text{ is open in } \operatorname{dom}(\delta_X)\},
\]
where the domain carries the subspace topology of \(\mathcal B\). This is the final topology of the partial surjection. Indeed,

\begin{known}[Computable total maps are continuous]\label{kn:computable-continuous}
	Let \(\mathbf X=(X,\delta_X)\) and\(\mathbf Y=(Y,\delta_Y)\) be represented spaces.
	If a total single-valued map \(f:X\to Y\) has a computable realizer, then
	\[
		f:(X,\tau_{\delta_X})\longrightarrow(Y,\tau_{\delta_Y})
	\]
	is continuous.
\end{known}

\begin{remark}[Why representations must be fixed in a Type-2 comparison]\label{rem:graph-representation}
	Let \(\mathbf X=(X,\delta_X)\) and \(\mathbf Y=(Y,\delta_Y)\) be represented spaces, and let
	\(f:X\to Y\) be any total function. Define another representation of \(X\) by
	\[
	\begin{split}
		\operatorname{dom}(\delta_X^f)
		&=
		\{\langle p,q\rangle:
		p\in\operatorname{dom}(\delta_X),\
		q\in\operatorname{dom}(\delta_Y),\
		\delta_Y(q)=f(\delta_X(p))\},\\
		\delta_X^f(\langle p,q\rangle)&=\delta_X(p).
	\end{split}
	\]
	This is a partial surjection: every \(x\in X\) has an original name \(p\), and \(f(x)\) has a \(\delta_Y\)-name \(q\). The computable second projection realizes
	\[
		f:(X,\delta_X^f)\longrightarrow\mathbf Y.
	\]
	Thus every such \(f\) can be made computable by changing the input representation. The first projection computes the identity \((X,\delta_X^f)\to\mathbf X\), but the converse translation need not be computable.
\end{remark}

\begin{definition}[Weihrauch reducibility]\label{def:weihrauch}
	For problems \(f\) and \(g\), write \(f\leq_{\rm W} g\) if there are partial computable functions
	\[
	H,K:\subseteq\mathcal B\to\mathcal B
	\]
	called witnesses in the following, such that
	\[
		p\longmapsto H\langle p,G(K(p))\rangle
	\]
	realizes \(f\) for every realizer \(G\vdash g\). Write \(f\equiv_{\rm W} g\) if both reductions hold.

    We write further $f\leq_{\mathrm{sW}}g$ if there are partial computable $H,K:\subseteq\mathcal B\to \mathcal B$ such that
\[
 H\circ G\circ K\vdash f \quad\text{for every realizer }G\vdash g.
\]
Write $f\equiv_{\mathrm{sW}}g$ if both strong reductions hold. An equivalence class under $\equiv_{\mathrm W}$ is called a \textit{Weihrauch degree}.

A represented problem $g$ is a \textit{cylinder} if
\[
 \operatorname{id}_{\mathcal B}\times g\leq_{\mathrm{sW}}g.
\]
\textbf{Caveat:} This use of "cylinder" is distinct from an exact-observation cylinder.
\end{definition}

For \(q\in\mathcal B\) define \(q_n\in\mathcal B\) by
\[
	q_n(j)=q(\langle n,j\rangle).
\]
The limit operation
\[
	\Lim:\subseteq\mathcal B\to\mathcal B
\]
is defined exactly when \((q_n)_{n\in\mathbb N_0}\) converges in the product topology of \(\mathcal B\), and its value is
\[
	(\Lim q)(j) = \lim_{n\to\infty}q(\langle n,j\rangle).
\]
The scalar limits on the right are taken in the discrete space \(\mathbb N_0\): equivalently, every coordinate sequence is eventually constant. We use ordinary composition to define
\[
	L_0:=\operatorname{id}_{\mathcal B}, \quad L_{k+1}:=\Lim\circ L_k \quad(k\in\mathbb N_0).
\]
Thus the domain of \(L_{k+1}\) consists of those \(q\in\operatorname{dom}(L_k)\) for which
\(L_k(q)\in\operatorname{dom}(\Lim)\). The notation \(L_k\) is used throughout to avoid confusion of ordinary iteration $f^{[k]}$ with the jump notation for Weihrauch problems $f^{(k)}$.

\begin{definition}[Finite limit rank]\label{def:lim-rank}
	For a problem \(f\) between represented spaces, put
	\[
		\operatorname{rank}_{\rm W}(f) := \min\{k\in\mathbb N_0:f\leq_W L_k\},
	\]
	with value \(\infty\) if this set is empty.
\end{definition}

Since Weihrauch reducibility is transitive, this rank depends only on the Weihrauch degree of \(f\). In particular, computably equivalent input and output representations give the same rank.

This theory is not new: Our concrete operation \(L_k\) represents the degree denoted by the \(k\)-fold compositional power \(\Lim^{[k]}\) in \cite[Sec.~11.6]{brattka2021weihrauch}.
For \(k\in \mathbb N\), \(L_k\) is the problem on names associated with the \((k-1)\)-fold jump of \(\Lim\). The correspondence with compositional powers follows from the standard realizer-version, jump, and cylinder results in \cite[Lem.~11.3.8, Defs.~11.6.8-11.6.9, and
Props.~11.6.6 and Prop.~11.6.14]{brattka2021weihrauch}.

\subsection{A concrete normal form on names}

Now we can state a concrete name-tower normal form that will be the crucial key to fix the comparison of SCI and TTE. For that define

\begin{definition}[Uniform towers on names]\label{def:nameTower}
	Let $\mathbf X=(X,\delta_X)$ and $\mathbf Y=(Y,\delta_Y)$ be
	represented spaces, $f:\subseteq X\rightrightarrows Y$ a problem and $k\in \mathbb{N}_0$. A \textit{computable name tower of height $k$} (for $f$) is a partial computable map
	\[
		A: \subseteq \mathbb{N}^k_0 \times \mathcal B \to \mathcal B
	\]
	where $\mathbb{N}^k_0 \times D_f \subseteq \dom(A)$ and the following holds: For every $p\in D_f$, all intermediate limits in
\[
    F_A(p):=\lim_{n_k\to\infty}\cdots\lim_{n_1\to\infty} A(n_k,\ldots,n_1,p)
\]
are required to exist in $\mathcal B$. We define $F_A$ on $D_f$ by this formula and require $F_A\vdash f$. The leaf streams and intermediate streams need not be valid
$\delta_Y$-names. The final stream must name an admissible output. At height $0$, the condition is simply $A \vdash f$, where we identify $\mathbb N^0_0 \times \mathcal B$ with $\mathcal B$.
\end{definition}

This definition is a natural extension to $\mathcal B$ of the concept of the Kleene-Shoenfield tower from \cite[Def.~3.5]{bensolvability}.

Related real-valued limit constructions appear in \cite[Ex.~27, Cor.~40]{Wyeth26}. There, finite approximants on names may depend on the chosen input name, while the final value is independent of that choice. \cite[Cor.~40]{Wyeth26} concerns real-valued functions on the specified rational or sequence domains and uses one jointly finitely computable rational-valued approximation map. Our formulation works directly with output names for arbitrary represented spaces.

\begin{lemma}[Coding of a name tower]\label{lem:codNameTower}
	A computable name tower of height $k\in\mathbb{N}_0$ for $f$ exists iff there is a partial computable $K:\subseteq \mathcal B \to \mathcal B$, such that $L_k \circ K \vdash f$.
\end{lemma}
\begin{proof}
	If $k=0$, we can simply take $K=A$. For $k\in \NN$, fix the computable bijection
	\[
		c_k(n_1,\ldots,n_k,j) = \langle n_1,\langle n_2, \ldots, \langle n_k,j\rangle \dots \rangle \rangle.
	\]
	Given $A$, produce its entire array by setting
	\[
		K(p)(c_k(n_1,\ldots,n_k,j)) = A(n_k,\ldots,n_1,p)(j).
	\]
	As $A$ is jointly computable and defined on all the required indices, it's easy to see that for every required input name $p$, every requested output coordinate can be computed. For the converse we note that the same equation extracts a jointly computable array from $K$. Now the ordering in $c_k$ forces $L_k$ to take the $n_1$-limit first up to taking the $n_k$-limit lastly. So the domain condition coincides with the existence of all intermediate $\mathcal B$-limits.
\end{proof}

We arrive at
\begin{theorem}[Name-tower normal form]\label{thm:NTNF}
	Let $\mathbf X,\mathbf Y$ be arbitrary represented spaces, $f:\subseteq X \rightrightarrows Y$ a problem, and $k\in \mathbb{N}_0$. Then the following are equivalent.
	\begin{enumerate}
		\item[(1)] $f\leq_{\rm W} L_k$;
		\item[(2)] there exists a partial computable $K:\subseteq \mathcal B \to \mathcal B$, such that $L_k \circ K \vdash f$;
		\item[(3)] $f$ admits a computable name tower of height $k$.
	\end{enumerate}
\end{theorem}
\begin{proof}
	The equivalence of (2) and (3) follows directly from \Cref{lem:codNameTower} and the implication (2)$\Rightarrow$(1) is trivial. All there is left to prove is therefore the implication (1)$\Rightarrow$(2). For this, choose partial computable witnesses $H,K_0$. Obviously for every $p \in D_f, \, K_0(p) \in \dom(L_k)$ and $p \mapsto H \langle p, L_k(K_0(p))\rangle$ realizes $f$ by definition. There is now a computable map $C_k : \mathcal B \times \mathcal B \to \mathcal B$, such that
	\[
		L_k(C_k(p,q)) = \langle p,L_k(q)\rangle
	\]
	for $q \in \dom(L_k)$. Indeed, for $k=0$ just use the ordinary pairing function. For $k\in \mathbb{N}$, define it on the coordinates $c_k$ by
	\begin{align*}
		C_k(p,q)(c_k(n_1,\ldots,n_k,2j)) = p(j),\\
		C_k(p,q)(c_k(n_1,\ldots,n_k,2j+1)) = q(c_k(n_1,\ldots,n_k,j)).
	\end{align*}
	By \cite[Cor.~2.8, Prop.~2.10]{brattka18}, each $L_k$ is transparent: For the partial computable $H$ there is a partial computable $T$, such that, whenever $r \in \dom(L_k)$ and $L_k(r) \in \dom(H)$, 
	\[
		T(r) \in \dom(L_k), \quad L_k(T(r)) = H(L_k(r)).
	\]
	If $k=0$ we can simply take $T=H$. By setting
	\[
		K(p) = T(C_k(p,K_0(p))),
	\]
	it's easy to see that this is a partial computable map, defined on all required input names, and
	\[
		L_k(K(p)) = H \langle p,L_k(K_0(p))\rangle,
	\]
	which implies $L_k \circ K \vdash f$.
\end{proof}

This was the preparation of the Weihrauch side. We illustrate now how this applies to bridge to a fixed SCI information interface. For that, let $\mathcal P$ be a computational problem for a countable $\Lambda$. Represent now $\operatorname{Ev}_{\Lambda}(\Omega) \subseteq \mathbb{C}^{\mathbb N_0}$ as the restriction of Cauchy representation of $\mathbb{C}^{\mathbb{N}_0}$. By \Cref{prop:factorization}, there exists a unique map
\[
	\Phi : \operatorname{Ev}_{\Lambda}(\Omega) \to \mathcal M, \quad \Phi(\operatorname{Ev}(A)) = \Xi(A).
\]
Assume now that $\mathcal M$ is equipped with a representation $\delta_{\mathcal M}$. Now we can define the

\begin{definition}[Name-based SCI]\label{def:name-sci}
	Let $\mathcal P=(\Xi,\Omega,(\mathcal M,d),\Lambda)$ be a computational problem with countable $\Lambda \neq \emptyset$, and use
	the fixed enumeration of \Cref{def:evaluation-table}. Equip
	\[
		I_\Lambda=\operatorname{Ev}_\Lambda(\Omega)\subseteq\mathbb C^{\mathbb N}
	\]
	with the restriction $\delta_{I_\Lambda}$ of the countable product of the standard Cauchy representation of $\mathbb C$. Put $\mathbf I_\Lambda=(I_\Lambda,\delta_{I_\Lambda})$.
	Fix a represented output space $\mathbf M=(\mathcal M,\delta_{\mathcal M})$.
	Take the unique map $\Phi:I_\Lambda\to\mathcal M$ from \Cref{prop:factorization} with
	\[
		\Phi(\operatorname{Ev}_\Lambda(A))=\Xi(A)\qquad(A\in\Omega).
	\]
	Define $\operatorname{SCI}_{\mathrm{name}}(\mathcal P)$ to be the least height of a uniform tower on names for $\Phi:\mathbf I_\Lambda\to\mathbf M$, with value $\infty$ if
	there is no such finite height.
\end{definition}

Here, we left the relation of $\delta_{\mathcal M}$ and the metric $d$ unspecified. This was name convergence need not correspond to convergence in the metric $d$. One possibility is to fix the Cauchy representation for the computable metric space $\mathcal M$. We decided to implement this needed specification locally in the corresponding results.

\begin{example}[Rationality testing: the name height]\label{ex:rationality-late}
	We continue with \Cref{ex:rationality-early}, now with standard Cauchy names for $\mathbb R$ and the standard discrete representation of $2$. The repeated identity evaluation table is computably isomorphic to $\mathbb R$: one direction repeats a Cauchy name, and
	the other takes the real part of its first coordinate. We claim that
	\[
		\operatorname{SCI}_{\mathrm{name}}(\mathcal P)=2.
	\]
	Fix a computable enumeration $(q_j)_{j\in\mathbb N}$ of $\mathbb Q$.
	From a Cauchy name $p$ of $x$, obtain rational approximants $r_t(p)$
	with $|r_t(p)-x|\leq2^{-t}$, $t\in\mathbb N$; a computable
	normalization supplies this error convention if necessary. Define for $n,s \in \mathbb N$ and $b \in 2$
	\[
		b_{n,s}(p)=1
		\quad\Longleftrightarrow\quad
		(\exists j\leq n)(\forall t\leq s)|r_t(p)-q_j|\leq2^{-t}.
	\]
	These bits are jointly computable in $(n,s,p)$. For fixed $n$,
	\[
	\lim_{s\to\infty}b_{n,s}(p)
	=\mathbf 1_{\{q_1,\ldots,q_n\}}(x).
	\]
	Indeed, an actual equality $x=q_j$ is never rejected, whereas every
	false equality is rejected at some finite stage, and there are only
	finitely many candidates. The outer limit in $n$ is $\mathbf 1_{\mathbb Q}(x)$.
	Replacing bits by their canonical names gives a jointly computable
	height-two name tower. Its indices are reindexed from $\mathbb N$
	to $\mathbb N_0$ when using the convention of the name-tower definition.
	
	If the name height were at most one, the induced function would be
	$\Sigma^0_2$-measurable by \cite[Thm.~11.6.5]{brattka2021weihrauch}. Thus its inverse image of the open singleton $\{0\}$, namely $\mathbb R\setminus\mathbb Q$, would be $F_\sigma$. Every closed subset of the irrationals has empty interior, so such an
	$F_\sigma$ representation, together with the rational singletons, would make $\mathbb R$ meagre in itself, contradiction to the Baire category theorem. Consequently, for this one fixed problem and the specified interfaces,
	\[
		\SCIG(\mathcal P)=0,
		\quad \SCIA(\mathcal P)=1,
		\quad \operatorname{SCI}_{\mathrm{name}}(\mathcal P)=2.
	\]
\end{example}

\begin{corollary}[Name-rank identity]\label{cor:nameRankId}
	In the setting of \Cref{def:name-sci},
	\[
		\operatorname{SCI}_{\mathrm{name}} (\mathcal P) = \operatorname{rank}_{\rm W}(\Phi).
	\]
\end{corollary}
\begin{proof}
	This follows directly from \Cref{thm:NTNF}.
\end{proof}	

In \cite[\S~7]{neumann2018topological} Neumann-Pauly only dealt with finite-real-inputs; we aim to generalize this now to arbitrary (countable) evaluation sets. This way, the probably first most natural mathematical question now might be what kind of topology the factorization of \Cref{prop:factorization} generates. Recall for this the exact-observation topology $\tau_\Lambda^{=}$ introduced before \Cref{prop:compact-currying}. This is exactly the initial topology of the maps $f : \Omega \to \mathbb{C}_{\rm d}$, where $\rm d$ denotes discrete (topology). Indeed, one can easily show that every general algorithm $\Gamma : \Omega \to \mathcal M$ is local constant w.r.t. $\tau^{=}_{\Lambda}$ and therefore continuous. Now define inductively for $k \in \mathbb{N}_0$
\[
\mathsf{BC}_0 := C((\Omega,\tau^{=}_{\Lambda}),\mathcal M), \quad \mathsf{BC}_{k+1} := \{\text{pointwise limits of sequences from } \mathsf{BC}_k\}.
\]
Then all together we get that
\[
\SCIG(\mathcal P) \leq k \Rightarrow \Xi \in \mathsf{BC}_k ((\Omega,\tau^{=}_{\Lambda}),\mathcal M),
\]
as the basis maps are continuous and every further limit level increases the possible $\mathsf{BC}$ class by at most $1$. By contraposition we derived the following possible SCI lower bound scheme:
\[
\Xi \notin \mathsf{BC}_k ((\Omega,\tau^{=}_{\Lambda}), \mathcal M) \Rightarrow \SCIG (\mathcal P) > k.
\]
This is the abstracted core mechanism behind later SCI results in \cite[Thm.~5.17, Rem.~5.18, Thm.~5.19]{Colbrook24Geo}, which are calibration results for alternating quantifier problems for bit sequences with coordinate requests. The reason is that on the cantor space $2^{\mathbb N}$ the initial topology $\tau^{=}_{\Lambda}$ and the product topology coincide.

\begin{proposition}[Lower bound characterization]\label{prop:lowerboundSCI}
	Suppose that $\Lambda \neq \emptyset$ is countable. Then
	\[
	\SCIG(\mathcal P) = 0 \iff \Xi \text{ is locally constant with respect to }\tau^{=}_{\Lambda}.
	\]
	Moreover for every $k\in \mathbb{N}$ we have
	\[
	\SCIG(\mathcal P) \leq k \iff \Xi \in \mathsf{BC}_k ((\Omega,\tau^{=}_{\Lambda}), \mathcal M).
	\]
\end{proposition}
\begin{proof}
	If $\Omega = \emptyset$, there is nothing to prove. We assume therefore in the following $\Omega \neq \emptyset$.
	
	The implication $\Rightarrow$ holds also for uncountable $\Lambda$ and is already proved by the argument before this proposition. We prove now the other implication.
	
	Enumerate $\Lambda=(f_i)_{i\geq1}$, allowing repetitions when $\Lambda$ is finite, and put
	\[
	C_n(A)=\{B\in\Omega:f_i(B)=f_i(A)\text{ for }1\leq i\leq n\}.
	\]
	These cylinders form a neighbourhood basis at $A$. For a locally constant map $h$, define
	\[
	\nu(A)=\min\{n\geq1:h\text{ is constant on }C_n(A)\}.
	\]
	This minimum exists. If $B\in C_{\nu(A)}(A)$, then $C_j(B)=C_j(A)$ for every $j\leq\nu(A)$, so $\nu(B)=\nu(A)$. Thus $h$, equipped with the query policy $\Lambda_h(A)=\{f_1,\ldots,f_{\nu(A)}\}$, is a general algorithm. This finishes the proof for height $0$.
	
	We next show that every continuous $h:(\Omega,\tau_\Lambda^{=})\to\mathcal M$ is uniformly approximable by general algorithms. For $\varepsilon>0$, put
	\[
	\nu_\varepsilon(A)=
	\min\{n\geq1:\operatorname{diam}h(C_n(A))<\varepsilon/2\}.
	\]
	Continuity of $h$ ensures that this minimum exists. As above, $B\in C_{\nu_\varepsilon(A)}(A)$ implies $\nu_\varepsilon(B)=\nu_\varepsilon(A)$; hence the distinct cylinders $C_{\nu_\varepsilon(A)}(A)$ form a partition of $\Omega$. Choose a representative $a_C\in C$ for every member $C$ of this partition, and define
	\[
	\Gamma_\varepsilon(A)=h(a_{C_{\nu_\varepsilon(A)}(A)}),\quad
	\Lambda_{\Gamma_\varepsilon}(A)=\{f_1,\ldots,f_{\nu_\varepsilon(A)}\}.
	\]
	Both the output and the query set are constant on the corresponding cylinder, so this is a general algorithm. Moreover
	\[
	\sup_{A\in\Omega}d(\Gamma_\varepsilon(A),h(A)) \leq\varepsilon/2<\varepsilon.
	\]
	
	Now if $\Xi\in\mathrm{BC}_1$, write $\Xi(A)=\lim_jh_j(A)$ with each $h_j$ continuous, and choose a general algorithm $\Gamma_j$ uniformly with error at most $2^{-j}$ of $h_j$. Then $\Gamma_j(A)\to\Xi(A)$ for every $A$, proving the reverse implication for $k=1$.
	For $k\geq2$, write $\Xi=\lim_jh_j$ with $h_j\in\mathrm{BC}_{k-1}$. By induction, choose a height-$(k-1)$ general tower for each $h_j$, padding shorter towers if necessary. Adding the outer index $j$ gives a height-$k$ general tower for $\Xi$. 
\end{proof}

The implication $\Leftarrow$ here is false for uncountable $\Lambda$; here is a possible counterexample.

\begin{example}[Countability is necessary for the converse]\label{ex:uncountable-counterexample}
	Put $I_0:=(-\infty,0]$ and $I_1:=(0,\infty)$, and let
	\[
	\Omega:=\bigl([I_0]^{<\omega}\cup[I_1]^{<\omega}\bigr) \setminus\{\varnothing\},\quad \Lambda:=\{\pi_t:t\in\mathbb R\}, \quad \pi_t(S):=\mathbf 1_S(t).
	\]
	Equip $\mathcal M=2$ with the discrete metric, and define $\Xi(S)=i$ when
	$S\subseteq I_i$. These data form a computational problem $\mathcal P$ because the full membership table determines $S$. Now the map $\Xi$ is locally constant in the exact-observation topology $\tau^{=}_{\Lambda}$. Indeed, if $t\in S\subseteq I_i$, the neighbourhood
	$\{T\in\Omega:\pi_t(T)=1\}$ of $S$ has constant value $i$. Nevertheless,
	\[
	\SCIG(\mathcal P)=\infty.
	\]
	
	To prove this, let $\Gamma$ be any general algorithm with a stable finite query policy, and put
	\[
	F(S)=\{t\in\mathbb R:\pi_t\in\Lambda_\Gamma(S)\}, \quad F_t=F(\{t\}).
	\]
	We first show that $t\notin F_t$ for some $t\in I_0$. Otherwise choose distinct $a_n\in I_0$, $n\in\mathbb N$, and then choose $b\in I_0\setminus\bigcup_nF_{a_n}$. The input $S_n=\{a_n,b\}$ gives the same answers as $\{a_n\}$ on $F_{a_n}$, so stability implies $F(S_n)=F_{a_n}$.
	
	If $a_n\notin F_b$, it also gives the same answers as $\{b\}$ on $F_b$, and hence $F(S_n)=F_b$. This is impossible because $b\in F_b$ but $b\notin F_{a_n}$. Thus $a_n\in F_b$ for every $n$, contradicting finiteness of $F_b$.
	
	Choose $t\in I_0$ with $t\notin F_t$. For every $S\in\Omega$ disjoint from $F_t$, all answers on $F_t$ agree with those of $\{t\}$, so the general-algorithm condition gives $\Gamma(S)=\Gamma(\{t\})$. Hence each leaf algorithm is constant on all inputs avoiding one finite set of addresses.
	
	A raw tower of finite height has only countably many leaf algorithms. Let $J\subseteq\mathbb R$ be the union of their finite exceptional address sets. Choose $u\in I_0\setminus J$ and $v\in I_1\setminus J$. Every leaf has identical values on $\{u\}$ and $\{v\}$, and all its iterated limits therefore have identical values on these two inputs. They cannot equal $\Xi$, which has
	values $0$ and $1$ respectively. The same argument includes a height $0$ tower, consisting of a single leaf.
\end{example}

\begin{lemma}[Uniform simulation of oracle BSS leaves on names]\label{lem:bss}
Let $k\in\mathbb N_0$, let $I_{\Lambda}\subseteq\mathbb C^{\mathbb N}$ carry the restriction of the standard product Cauchy representation,
and let $(\mathcal M,d,\alpha)$ be a nonempty computable metric space.

Suppose one parameter-free deterministic oracle BSS machine, on every input $(\vec n,t)\in\mathbb N^k\times I_{\Lambda}$, halts with a finite real tuple $u(\vec n,t)$.
Assume the oracle and instruction conventions of \Cref{def:BSS-arith-alg}. Let
\[
 \delta_{\mathcal M}^{\mathrm{fin}}:\subseteq\mathbb R^*\to \mathcal M
\]
be Type-2 computable into the Cauchy representation of $\mathcal M$. Assume $u(\vec n,t)\in\operatorname{dom}(\delta_{\mathcal M}^{\mathrm{fin}})$ for every $(\vec n,t)$. Then the total joint map
\[
 \Theta:\mathbb N^k\times I_{\Lambda} \to \mathcal M,\quad
 \Theta(\vec n,t):=\delta_{\mathcal M}^{\mathrm{fin}}(u(\vec n,t))
\]
satisfies $\Theta\leq_{\mathrm W} L_1$. At $k=0$, $\mathbb N^0$ denotes a singleton.
\end{lemma}

\begin{proof}
Fix an effective enumeration $(\mathcal T_e)_{e\in\mathbb N_0}$ of oracle Turing programs and define
\[
 J(p)(e):=
 \begin{cases}
 1,&\mathcal T_e^p\text{ halts on empty input},\\
 0,&\text{otherwise}.
 \end{cases}
\]
We construct a partial computable functional $V$ such that $V(\vec n,p,J(p))$ is a name of $\Theta(\vec n,t)$ whenever $p$ is a valid name of $t\in I_{\Lambda}$.

Effectively code finite expressions formed from rational constants, the variables
\[
 Z_{i,0}:=\operatorname{Re}(t_i),\quad
 Z_{i,1}:=\operatorname{Im}(t_i)\quad(i\in\mathbb N),
\]
and $+,-,\cdot,/$. For a code $c$, let $E_c(t)$ be the expression's value when every division occurring in its evaluation is defined. Structural recursion on expressions, using partial computability of division on nonzero denominators, gives a partial computable rational-valued procedure $A$ such that
\[
 |A(c,p,s)-E_c(t)|\leq 2^{-s} \quad(s\in\mathbb N_0)
\]
whenever $p$ names $t$ and $E_c(t)$ is defined. On these inputs $A(c,p,s)$ is defined for every $s$.

Uniformly from $c$, compile oracle programs with indices $e_+(c)$ and $e_-(c)$ that search, respectively, for
\[
 A(c,p,s)>2^{-s},\quad -A(c,p,s)>2^{-s}.
\]
For a defined expression these searches halt exactly when $E_c(t)>0$ and $E_c(t)<0$, respectively.
The forward implications follow from the error estimate, and a strictly positive or negative value eventually passes the respective test. Hence the two bits of $J(p)$ decide its sign, including equality to $0$ when both bits are $0$.

Simulate the BSS machine symbolically, storing its control state and an expression code for each register accessed so far.
Initial register values, including the finite integer tuple $\vec n$, have rational-constant codes. Arithmetic and copy instructions update these codes effectively;
register shifts update their positions. At a branch on the register coded by $c$, use the two bits $J(p)(e_+(c))$ and $J(p)(e_-(c))$.

At a QUERY instruction let $c$ code its address register. Its value is a positive integer by hypothesis. Computing $A(c,p,3)$ and rounding to the nearest integer recovers
that address, since its error is at most $1/8$. The answer registers receive codes $Z_{i,0},Z_{i,1}$ for that address.

Induction over the executed instructions proves that the simulated control state is the actual state and that each stored code evaluates to the value of the corresponding actual register. In particular all expressions evaluated along the simulation are defined, and the simulation reaches the actual halting state. A code selected using earlier jump answers is still a finite integer; its compiled indices are submitted to the same oracle $J(p)$.

Further, at termination the simulation has the output length $\ell$ and expression codes $c_1,\ldots,c_\ell$ for $u(\vec n,t)$. The procedures $A(c_j,p,s)$ yield a name of this finite tuple, uniformly relative to $(p,J(p))$. Apply a fixed computable realizer of $\delta_{\mathcal M}^{\mathrm{fin}}$. The domain promise ensures that it produces an $\mathcal M$-name. This defines the claimed $V$.

Let $q$ name $(\vec n,t)$ and extract $\vec n$ and a table name $p$ computably from $q$. Define the computable preprocessor $K$ by
\[
 K(q)(\langle s,e\rangle):=
 \begin{cases}
 1,&\mathcal T_e^p\text{ halts within }s\text{ steps},\\
 0,&\text{otherwise}.
 \end{cases}
\]
Every coordinate is obtained by a finite oracle simulation. For each $e$ the sequence in $s$ is eventually constant, so
\[
 K(q)\in\operatorname{dom}(L_1),\qquad
 L_1(K(q))=J(p).
\]
The partial computable postprocessor
\[
 H\langle q,z\rangle:=V(\vec n,p,z)
\]
therefore returns a name of $\Theta(\vec n,t)$ on $z=L_1(K(q))$. These $H,K$ witness $\Theta\leq_{\mathrm W}L_1$.
\end{proof}

\begin{lemma}[Uniform leaf cost for metric towers]\label{lem:uniform-leaf-cost}
Let $X$ be a represented space, let $(\mathcal M,d,\alpha)$ be a nonempty computable metric space with its Cauchy representation, and let $k,r\in\mathbb N_0$. Suppose the total joint map
\[
 \Theta:\mathbb N^k\times X\to \mathcal M
\]
satisfies $\Theta\leq_{\mathrm W}L_r$. Assume that for every $x\in X$ all intermediate metric limits in
\[
 f(x):=\lim_{n_k\to\infty}\cdots\lim_{n_1\to\infty}
 \Theta(n_k,\ldots,n_1,x)
\]
exist in $\mathcal M$. Then $f\leq_{\mathrm{sW}}L_{k+r}$. For $k=0$ read $f=\Theta$ with no index input.
\end{lemma}

\begin{proof}
First consider the partial map
\[
 \operatorname{Lim}_{\mathcal M}:\subseteq \mathcal M^{\mathbb N}\to \mathcal M, \quad (y_n)\mapsto\lim_n y_n,
\]
whose domain consists of the sequences converging to a point of $\mathcal M$. Let $q$ name such a sequence and let $q_n$ be its component names.
For $j\in\mathbb N_0$ put $\varepsilon_j:=2^{-(j+2)}$ and define
\[
 B_j(q):=\{N\in\mathbb N:
   (\exists n,m\geq N)\ d(y_n,y_m)>\varepsilon_j\}.
\]
This set is computably enumerable relative to $q$, uniformly in $j$, because distance is computable and strict inequality is semidecidable. Choose uniformly computable increasing finite approximations $B_{j,s}(q)$ with union $B_j(q)$, and put
\[
 N_{j,s}:=\min(\mathbb N\setminus B_{j,s}(q)),\quad A(s,q)(j):=q_{N_{j,s}}(j+2).
\]
The Cauchy property implies $\mathbb N\setminus B_j(q)\ne\varnothing$. Let $N_j$ be its least member. All smaller integers eventually enter
$B_{j,s}(q)$ and $N_j$ never does, so $N_{j,s}$ is eventually $N_j$. For $y:=\lim_n y_n$, continuity of distance gives $d(y_{N_j},y)\leq\varepsilon_j$. Consequently
\[
 d\bigl(\alpha(q_{N_j}(j+2)),y\bigr)
 \leq 2^{-(j+2)}+\varepsilon_j\leq 2^{-j}.
\]
Thus $\lim_s A(s,q)$ is a Cauchy name of $y$. Packing $A(s,q)(j)$ into the $(s,j)$ coordinates gives a computable preprocessor witnessing $\operatorname{Lim}_{\mathcal M} \leq_{\mathrm{sW}}L_1$.

Countably many parallel $L_a$ computations cost at most $a$ limits: in the coding of \Cref{lem:codNameTower}, put the extra computation index into the final output-coordinate slot. For example, pack a uniformly computable family $K(e,p)$ as
\[
 V(p)\bigl(c_a(t_1,\ldots,t_a,\langle e,j\rangle)\bigr)
 :=K(e,p)\bigl(c_a(t_1,\ldots,t_a,j)\bigr).
\]
The same $a$ limits then act on every $(e,j)$ coordinate. All intermediate streams exist because the corresponding coordinates exist in every individual computation. At $a=0$ use computable tupling.

Costs also add under composition of compatible single-valued maps. By \Cref{thm:NTNF} their realizers can be written $L_aU$ and $L_bV$ with partial computable $U,V$.
Transparency of $L_b$ gives partial computable $T$ with $U L_b=L_bT$ on the required domain. Hence
\[
 L_aU L_bV=L_a L_bTV=L_{a+b}TV.
\]
At $b=0$ this is ordinary computable composition. The transparency statement for partial maps and its closure under composition are established in \cite[Cor.~2.8 and Prop.~2.10]{brattka18}.

Apply now \Cref{thm:NTNF} to the joint map $\Theta$ and run its preprocessor on every tuple $\vec n$ in parallel. The preceding coding yields
\[
 G:X\to \mathcal M^{\mathbb N^k},\quad
 G(x):=(\Theta(\vec n,x))_{\vec n\in\mathbb N^k}, \quad G\leq_{\mathrm W}L_r.
\]
Apply $\operatorname{Lim}_{\mathcal M}$ in parallel to all innermost sequences, and repeat for each of the $k$ limit levels. Each level costs at most one limit and its input lies in the required domain by hypothesis. Additivity gives $f\leq_{\mathrm W}L_{r+k}$. \Cref{thm:NTNF} supplies a normal-form realizer $L_{r+k}W\vdash f$ with partial computable $W$; using identity output postprocessing gives the asserted strong reduction.
\end{proof}

\begin{theorem}[Meta Bridge Theorem]\label{thm:comparison}
Let $\mathcal P=(\Xi,\Omega,(\mathcal M,d),\Lambda)$ be a computational problem, where $\Lambda$ is countable and nonempty.
Fix a surjective enumeration $(f_i)_{i\in\mathbb N}$, allowing repetitions, and put
\[
    \operatorname{Ev}_\Lambda(A):=(f_i(A))_{i\in\mathbb N},\quad
    I_\Lambda:=\operatorname{Ev}_\Lambda(\Omega).
\]
Let $\delta_{I_\Lambda}$ be the restricted standard product Cauchy representation, and let $\Phi:I_\Lambda \to \mathcal M$ be the unique map with
$\Xi=\Phi\circ\operatorname{Ev}_\Lambda$ by \Cref{prop:factorization}.
\begin{enumerate}
\item[(a)] Suppose $\delta_\mathcal M:\subseteq \mathcal B \to \mathcal M$ is a representation continuous into the topology induced by $d$.
For every $r\in\mathbb N$,
\[
 \Phi\leq_{\mathrm W}L_r \iff
 \SCIG(\mathcal P)\leq r.
\]
If $\Phi$ is computable, then $\SCIG(\mathcal P)\leq1$. Thus
\[
    \SCIG(\mathcal P) \leq\max\{1,\operatorname{rank}_{\mathrm W}(\Phi)\},
\]
with $\max\{1,\infty\}:=\infty$.

\item[(b)] Suppose $(\mathcal M,d)$ has a specified nonempty computable metric structure and $\delta_\mathcal M$ is its Cauchy representation. Let $k,r\in\mathbb N_0$ and let a height $k$ raw tower for $\mathcal P$ have induced leaf maps $\widehat\Gamma_{\vec n}:I_\Lambda \to \mathcal M$.
If the total joint map
\[
    \Theta:\mathbb N^k\times I_\Lambda \to \mathcal M,\quad
    \Theta(\vec n,t):=\widehat\Gamma_{\vec n}(t)
\]
satisfies $\Theta\leq_{\mathrm W}L_r$, then $\Phi\leq_{\mathrm{sW}}L_{k+r}$.
At $k=0$ read this as $\Theta=\Phi$.

\item[(c)] Assume the output conventions of (b), and assume that the fixed finite decoder $\delta_\mathcal M^{\mathrm{fin}}: \subseteq\mathbb R^*\to \mathcal M$ is Type-2 computable into $\delta_\mathcal M$. For every $k\in\mathbb N_0$,
\[
 \SCIA(\mathcal P)\leq k \iff \Phi\leq_{\mathrm{sW}}L_{k+1}.
\]
\end{enumerate}
\end{theorem}

\begin{proof}
If $\Omega=\varnothing$, (a) is immediate and all reductions have empty input domain. Assume henceforth $\Omega\ne\varnothing$.

For (a), let $\tau$ be the subspace topology on $I_\Lambda \subseteq(\mathbb C_{\mathrm{d}})^{\mathbb N}$. The map $\operatorname{Ev}_\Lambda: (\Omega,\tau_\Lambda^{=})\to (I_\Lambda,\tau)$ is continuous. Choose a naming selector $\eta:\mathbb C\to\operatorname{dom}(\delta_{\mathbb C})$ with $\delta_{\mathbb C}\eta=\operatorname{id}_{\mathbb C}$. Define
\[
    \sigma(t)(\langle i-1,j\rangle):=\eta(t_i)(j) \quad(i\in\mathbb N,\ j\in\mathbb N_0).
\]
Then $\delta_{I_\Lambda}\sigma=\operatorname{id}_{I_\Lambda}$, and $\sigma:(I,\mathcal T)\to\mathbb B$ is continuous. Indeed, each finite output prefix depends on only finitely many exact coordinates of $t$. Put $S:=\sigma(I)$ with its subspace topology. This space is separable, zero-dimensional and metrizable.

Suppose $r\geq1$ and $\Phi\leq_{\mathrm W}L_r$. \Cref{thm:NTNF} supplies a height $r$ name tower with final realizer $\mathcal R$. Restricting its leaves and intermediate limits to $S$ gives $\mathcal R|_S\in\mathsf{BC}_r(S,\mathcal B)$. For every clopen cylinder $U\subseteq\mathcal B$, $\{p\in S: \mathcal R(p)\in U\}$ belongs to $\Delta^0_{r+1}(S)$. Indeed, this follows inductively from eventual membership in a clopen set for a convergent sequence, using
\[
 h^{-1}(U)=\bigcup_N\bigcap_{n\geq N}h_n^{-1}(U)
\]
and the same formula for its complement. As every open subset of $\mathcal B$ is a countable union of clopen cylinders, $\mathcal R|_S$ is $\Sigma^0_{r+1}$-measurable.

Put $D_\mathcal M:=\operatorname{dom}(\delta_\mathcal M)$ and $g:=\delta_\mathcal M\circ \mathcal R|_S$. The realizer condition gives $\mathcal R(S)\subseteq D_\mathcal M$. For metric-open $V\subseteq \mathcal M$, continuity of $\delta_\mathcal M$ gives open $W\subseteq\mathcal B$ such that $\delta_\mathcal M^{-1}(V)=D_\mathcal M\cap W$. Hence
\[
    g^{-1}(V)=(\mathcal R|_S)^{-1}(W)\in\Sigma^0_{r+1}(S).
\]
The space $\mathcal M$ is separable: $D_\mathcal M$ is second countable, and its continuous image under the surjection $\delta_\mathcal M$ is $\mathcal M$.
The Borel-Baire correspondence on separable zero-dimensional metrizable domains gives $g\in\mathsf{BC}_r(S,\mathcal M)$, see e.g. \cite[Thms.~24.3 and~24.10]{Kec95}. Continuous precomposition preserves these sequential classes, so
\[
    \Xi=g\circ\sigma\circ\operatorname{Ev}_\Lambda \in\mathsf{BC}_r((\Omega,\tau_\Lambda^{=}),\mathcal M).
\]
\Cref{prop:lowerboundSCI} gives $\SCIG(\mathcal P)\leq r$. If $\Phi$ is computable, choose a computable realizer $\mathcal R$ instead. Then $g$ and $\Xi$ are continuous, and \Cref{prop:lowerboundSCI} gives $\SCIG(\mathcal P)\leq1$. The rank inequality follows, including its tautological infinite case.

For (b), agreement on the entire evaluation table forces agreement of each raw leaf output. It therefore also forces agreement of every successive limit. The induced tower on $I_\Lambda$ has final target $\Phi$. Apply now \Cref{lem:uniform-leaf-cost} with $X=I_\Lambda$.

For (c), a witnessing arithmetic tower of height $j\leq k$ has one BSS machine for its entire indexed leaf family. \Cref{lem:bss} gives then $\Theta\leq_{\mathrm W}L_1$, so (b) yields $\Phi\leq_{\mathrm{sW}}L_{j+1}$. If the tower has a smaller height $j\le k$, the conclusion first gives $\Phi\le_{\mathrm{sW}} L_{j+1}$, and dummy limit indices pad this to $L_{k+1}$.
\end{proof}

\begin{corollary}[Consequences and the exceptional first level]
	\label{cor:comparison}
	Under the Cauchy-output assumptions of \Cref{thm:comparison}:
	\begin{enumerate}
		\item[(i)] If $m\ge2$ and $\SCIG(\mathcal P)\ge m$, then
		$\Phi\not\le_W L_{m-1}$. If $\SCIG(\mathcal P)=\infty$, then $\operatorname{rank}_W(\Phi)=\infty$.
		\item[(ii)] If $\SCIG(\mathcal P)=k\ge2$ and a height-$k$ raw tower has jointly
		Type-2-computable leaves, then $\operatorname{rank}_W(\Phi)=k$.
		\item[(iii)] If the computable-decoder arithmetic implementation is available and
		$\SCIG(\mathcal P)=\SCIA(\mathcal P)=k\ge2$, then
		\[
		k\le\operatorname{rank}_W(\Phi)\le k+1.
		\]
	\end{enumerate}
\end{corollary}

\begin{proof}
	For (i), a rank $r<m$ would give raw height at most $\max\{1,r\}<m$.
	The same argument rules out every finite rank when raw height is infinite.
	For (ii), part (b) of the theorem with $r=0$ gives rank at most $k$, and
	(i) rules out rank below $k$. For (iii), use (i) and part (c).
\end{proof}

\begin{remark}[What cannot be omitted]
	The factor $\max\{1,\cdot\}$ is necessary: the Cantor identity with bit
	evaluations is computable but has raw height one. The output topology in
	(a) matters: an arbitrary discontinuous representation may encode a
	non-Borel raw target into its output decoding. The effective finite decoder
	in (c) matters independently of parameter-freeness. Finally, separate
	computability of all leaves is insufficient for (b), as shown by the
	manuscript's nonuniform-prefix example.
\end{remark}

\subsection{An application of the meta bridge theorem}
To consolidate the abstract insights of \Cref{thm:comparison} and \Cref{cor:comparison} we give an classification example in this section. It can be seen as an computably reindexed version (provided, the bijections below are chosen computably) of \Cref{lem:alternating-calibration}. We equip therefore $\Omega_{\mathrm{mat}}$ with its standard product representation and $2$ with its discrete representation.

\begin{lemma}[Represented calibration of the alternating predicates]\label{lem:alternating-represented-calibration}
In the situation as above, for every $m\in\mathbb N$,
\[
 \operatorname{rank}_{\mathrm W}(\Xi_m)=m.
\]
\end{lemma}

\begin{proof}
First we note that the bounded-quantifier leaves in \Cref{lem:alternating-calibration} are jointly computable. Indeed, the finite rectangle is effectively enumerable and $\beta_m$ is computable uniformly in its arguments. The standard bit-array representation and its evaluation-table representation are computably isomorphic. \Cref{lem:uniform-leaf-cost} with $r=0$ gives now $\Xi_m\leq_{\mathrm{sW}}L_m$.

Suppose $\Xi_m\leq_{\mathrm W}L_j$ for $j<m$. By \Cref{thm:NTNF} and the computable continuous map taking a bit array to its canonical product name, the two fibers of $\Xi_m$ belong to $\Delta^0_{j+1}(\Omega_{\mathrm{mat}})$. For $j=0$ they are clopen; for $j>0$ this is the clopen-preimage induction used in the proof of \Cref{thm:comparison}(a). Consequently $\Xi_m^{-1}(\{1\})\in\Pi^0_m$, contrary to the $\Sigma^0_m$-completeness and non-$\Pi^0_m$ property established in \Cref{lem:alternating-calibration}. Thus no $j<m$ suffices.
\end{proof}

\begin{theorem}[One additional limit is necessary at every positive height]\label{thm:sharp}
Put $\Omega_{\mathbb R}:=\mathbb R^{\mathbb N^2}$ and let $\Lambda_{\mathbb R}$ consist of its real coordinate projections, viewed as complex-valued evaluations, with the fixed computable coordinate enumeration. Give the output $2$ its standard representation, and use the finite decoder
\[
 \operatorname{dom}(\delta_{2}^{\mathrm{fin}}):=\{(0),(1)\},\quad
 \delta_{2}^{\mathrm{fin}}((b)):=b.
\]
For $k\in\mathbb N$ define
\[
 R_k(z):=
 \begin{cases}z=0,&k\text{ odd},\\z\ne0,&k\text{ even},\end{cases}
\]
for $z \in \mathbb R$, and $\Psi_k:\Omega_{\mathbb R}\to 2$ by
\[
 \Psi_k(x)=1 \iff (Q_1n_1)\cdots(Q_kn_k) R_k(x(\beta_k(n_1,\ldots,n_k))).
\]
Let $\mathcal P_k$ be this computational problem and let $\Phi_k$ be its represented evaluation-table factor. Then
\[
 \SCIG(\mathcal P_k)=\SCIA(\mathcal P_k)=k, \quad
 \Phi_k\equiv_{\mathrm{sW}}\Xi_{k+1},\quad
 \operatorname{rank}_{\mathrm W}(\Phi_k)=k+1.
\]
\end{theorem}

\begin{proof}
Taking real parts, and conversely adjoining $0$ imaginary parts, gives a computable isomorphism between real arrays and their represented complex evaluation tables. We work with array names. The finite decoder is computable: on its promised domain one approximation with error less than $1/4$ determines the bit.

Replace each quantifier $Q_i n_i$ by $Q_i n_i\leq s_i$. A single parameter-free BSS machine enumerates the finite rectangle, performs the exact $0$ tests and evaluates the resulting Boolean expression. It outputs $(0)$ or $(1)$. Its fixed query set is the rectangle of coordinates and is a stable query policy. For fixed outer bounds, the innermost bounded quantifier stabilizes as its bound tends to $\infty$. Finite Boolean operations commute with that stabilization. Repeating at every level gives all
intermediate limits and the target $\Psi_k$. Relabel $n_{k-i+1}=s_i$ for the inner-first tower convention. Thus $\SCIA(\mathcal P_k)\leq k$.

For $b\in\Omega_{\mathrm{mat}}$ we put
\[
 x_b(a,c):=
 \begin{cases}
 1-b(a,c),&k\text{ odd},\\
 b(a,c),&k\text{ even}.
 \end{cases}
\]
Then $\Psi_k(x_b)=\Xi_k(b)$. Restricting any raw leaf to these inputs replaces each queried real coordinate by its corresponding bit coordinate. Agreement on the
restricted query set implies agreement on the original set, hence both output equality and equality of the chosen query sets. The restricted policy is finite, nonempty and stable. A height $j$ raw tower for $\Psi_k$ would therefore yield one for $\mathcal P_k^{\mathrm{mat}}$. Now \Cref{lem:alternating-calibration} gives $j\geq k$. Together with \Cref{thm:monotone-G-vs-BSS-A} this proves both SCI equalities.

Given a valid array name $p$, compute rational approximants $r_{v,t}(p)$ with $|r_{v,t}(p)-x(v)|\leq2^{-t}$, uniformly for $v\in\mathbb N^2$ and $t\in\mathbb N$.
Define a bit matrix $B_p$ by
\[
 B_p(\beta_{k+1}(\vec n,t)):=
 \begin{cases}
 \mathbf1_{\{|r_{\beta_k(\vec n),t}(p)|\leq2^{-t}\}},
       &k\text{ odd},\\
 \mathbf1_{\{|r_{\beta_k(\vec n),t}(p)|>2^{-t}\}},
       &k\text{ even}.
 \end{cases}
\]
This is computable from $p$, since the comparisons are rational. For odd $k$,
\[
 x(\beta_k(\vec n))=0 \iff (\forall t)\ B_p(\beta_{k+1}(\vec n,t))=1;
\]
for even $k$, nonzeroness is equivalent to the corresponding existential assertion. In both cases the added quantifier is $Q_{k+1}$, so $\Psi_k(x)=\Xi_{k+1}(B_p)$.
Identity output postprocessing proves $\Psi_k\leq_{\mathrm{sW}}\Xi_{k+1}$.

Conversely, given $b\in\Omega_{\mathrm{mat}}$, define a real array by
\[
 x_b(\beta_k(\vec n)):=
 \begin{cases}
 \displaystyle\sum_{t=1}^{\infty}2^{-t}
       (1-b(\beta_{k+1}(\vec n,t))),&k\text{ odd},\\[1ex]
 \displaystyle\sum_{t=1}^{\infty}2^{-t}
       b(\beta_{k+1}(\vec n,t)),&k\text{ even}.
 \end{cases}
\]
The tail after $T$ terms is at most $2^{-T}$, so a Cauchy name of $x_b$ is computable uniformly from $b$. Nonnegativity of the summands gives zero exactly when all the final bits are one in the odd case, and nonzero exactly when some final bit is one in the even case. Thus $\Psi_k(x_b)=\Xi_{k+1}(b)$ and identity output postprocessing
proves the reverse strong reduction. The evaluation-table isomorphism transfers both reductions to $\Phi_k$. \Cref{lem:alternating-represented-calibration} gives its rank.
\end{proof}

\begin{example}[A matrix with a row containing no zero]\label{ex:matrix}
The case $k=2$ can be formulated especially easy. For $x\in\mathbb R^{\mathbb N^2}$ put
\[
 \Psi_2(x)=1 \iff (\exists i)(\forall j)\ x(i,j)\ne0.
\]
Its arithmetic leaves are
\[
 \Gamma_{n,m}(x):=
 \max_{1\leq i\leq n}\min_{1\leq j\leq m} \mathbf1_{\{x(i,j)\ne0\}},\quad
 \Psi_2=\lim_n\lim_m\Gamma_{n,m}.
\]
\Cref{thm:sharp} then gives
\[
 \SCIG (\mathcal P_2)=\SCIA (\mathcal P_2)=2,\quad \Phi_2\equiv_{\mathrm{sW}}\Xi_3,\qquad
 \mathrm{SCI}_{\mathrm{name}}(\mathcal P_2)=\operatorname{rank}_{\mathrm W}(\Phi_2)=3.
\]
Here, the independently readable real-matrix input differs from the finite-real-tuple convention in \cite[Thm.~44]{neumann2018topological}.
\end{example}

\begin{caveat}
For every binary-valued represented problem $f$ and every $r\in\mathbb N_0$, one has $L_r\not\leq_{\mathrm{sW}}f$.
Indeed, choose a realizer of $f$ returning one fixed canonical name for each bit. A strong output postprocessor then has at most two possible outputs, whereas $L_r$ is surjective onto $\mathcal B$. Thus the rank equality above does not identify the strong degree with that of $L_{k+1}$.
\end{caveat}

\subsection{An unrestricted-postprocessing obstruction}
We will now see that raw height alone provides no finite upper bound. Here, we will need the following result that is well known as it follows directly from induction and \cite[Thm.~11.6.5]{brattka2021weihrauch}.

\begin{known}[Finite $\mathrm{rank}_{\mathrm W}$ implies Borel regularity]\label{kn:borel-regularity}
	Let $X,Y$ be computable metric spaces equipped with their Cauchy representations, and let $f:X\to Y$ be a total function. If $f \leq_{\rm W} L_n$ for some $n\in \mathbb N_0$, then $f$ is Borel measurable.
\end{known}

\begin{theorem}[Type-G can have $\mathrm{SCI}_G=0$ while $\operatorname{rank}_{\rm W}=\infty$]\label{thm:scig0-but-tteinf}
	Assume a standard represented space structure on $X=\C$ (e.g.\ Cauchy representation),
	and let $Y=2$ be discrete.
	Let $A\subseteq \C$ be a fixed analytic and non-Borel set (this exists as $\mathbb C$ is an uncountable Polish space, see then e.g. \cite[Cor.~26.2]{Kec95}).
	Define the single-valued map
	\[
	f:\C\to2,\qquad f(z):=\mathbf{1}_A(z).
	\]
	Consider the computational problem
	\[
	\mathcal{P}=(\Xi,\Omega,(\mathcal{M},d),\Lambda)
	\quad\text{with}\quad
	\Omega=\C,\ \mathcal{M}=2,\ \Xi=f,\ \Lambda=\{\mathrm{id}_\C\}.
	\]
	Then
	\begin{enumerate}
		\item $\mathrm{SCI}_G(\mathcal{P})=0$;
		\item $\operatorname{rank}_{\rm W}(f)=\infty$.
	\end{enumerate}
\end{theorem}

\begin{proof}
	\textbf{(1):} Since $f=\mathbf{1}_A\circ \mathrm{id}_\C$, this is a finite-query factorisation
	through the single evaluation $\mathrm{id}_\C\in\Lambda$ with post-processing $G=\mathbf{1}_A$.
	Thus $\mathrm{SCI}_G(\mathcal{P})=0$ by \Cref{thm:finite-factor-scig0}.
	
	\textbf{(2):} Suppose for contradiction that $f\le_{\mathrm W} L_n$ for some $n\in \mathbb N_0$.
	Then by \Cref{kn:borel-regularity}, the set $A=f^{-1}(\{1\})$ must be Borel,
	contradicting the choice of $A$. Hence $f\not\le_{\mathrm W} L_n$ for all $n\in \mathbb N_0$, so $\operatorname{rank}_{\rm W}(f)=\infty$.
\end{proof}

\begin{corollary}[No bound from raw type-G height to Weihrauch rank]\label{cor:no-bound-G-to-W}
	There is no function $h:\mathbb N_0\to\mathbb N_0$ such that, for every computational problem and represented enrichment in \Cref{def:name-sci},
	\[
		\SCIG(\mathcal P)\le n \ \Longrightarrow\ \operatorname{rank}_{\mathrm W}(\Phi)\le h(n) \quad(n\in\mathbb N_0).
	\]
\end{corollary}
\begin{proof}
	Take the problem of \Cref{thm:scig0-but-tteinf}. Enumerate its sole evaluation by $f_i=\operatorname{id}_{\mathbb C}$. The map $z\mapsto(z,z,\ldots)$ is a computable isomorphism from $\mathbb C$ onto $I_\Lambda$, with inverse being the first projection.
	Thus $\Phi$ is Weihrauch-equivalent to $f$ and has $\operatorname{rank}_{\rm W}(f) = \infty$, whereas $\SCIG(\mathcal P)=0$. No finite value of $h(0)$ can satisfy the asserted implication.
\end{proof}

We will see now that raw finite-information dependence already yields local constancy in the exact-observation topology $\tau^{=}_{\Lambda}$ (defined below). Descriptive-set lower bounds in a different, independently chosen topology require compatibility between that topology and the admitted base maps.

\section{Features and boundaries of Hansen's definition}{\label{sec:FeaturesHansen}}

This section makes explicit two features and more the logical boundary of the raw SCI more explicitly. First, exact finite information determines a topology that yields Baire-class lower bounds. Second, individual realizability of raw approximants does not by itself provide the uniform name approximations characterized in \Cref{thm:NTNF}. The lower-bound mechanism we extract(ed) has important precedents in \cite[Thm.~5.17, Rem.~5.18, Thm.~5.19]{Colbrook24Geo}.

We now use the represented structures of \Cref{sec:ReprSpLimHier}. The next distinction is between individual realizability of raw approximants and the uniform name-tower condition.

\begin{known}[Exact comparison depends on the interface]
	In the oracle BSS model of \Cref{def:BSS-arith-alg} with the identity evaluation on $\mathbb R$ and the identity decoder on the three one-real codes $(-1),(0),(1)$, the sign of an input can be computed with at most two exact order tests: first test $x<0$, and otherwise test $0<x$.
	With Cauchy names for $\mathbb R$ and the discrete representation of $\{-1,0,1\}$, the sign function is not Type-2 computable: it is discontinuous, whereas computable total maps are continuous by \Cref{kn:computable-continuous}.
\end{known}

Thus exact comparison supplies a base operation absent from ordinary Type-2 computation on Cauchy names. Comparisons between the two frameworks must specify the evaluation interface and any additional oracle operations being charged. The low-height BSS/TTE distinction
is also reflected in \cite[Obs.~43, Thm.~44]{neumann2018topological}.

We densify our results to the point that the standard SCI tower should be read as a raw extensional limit skeleton; to obtain a computability model, one must impose uniform implementation.

The next two constructions isolate a failure of uniformity: their deepest approximants are individually computable, while no computable name tower exists at the indicated height.

For that, recall that every general algorithm $\Gamma:\Omega\to \mathcal M$ induces a unique
$\widehat\Gamma:I_\Lambda\to \mathcal M$ with $\widehat\Gamma\circ\operatorname{Ev}_\Lambda=\Gamma$ by essentially (the proof of) \Cref{prop:factorization}. For fixed represented data, a raw height $k$ tower has \textit{individually computable leaves} if
\[
 (\forall\vec n\in\mathbb N^k)(\exists F_{\vec n})\quad
 F_{\vec n}\text{ is partial computable } \land \,
 F_{\vec n}\vdash\widehat\Gamma_{\vec n}.
\]
Its leaves are \textit{jointly computable} if one partial computable functional $F$ satisfies
\[
 \delta_\mathcal M(F(\vec n,p))= \widehat\Gamma_{\vec n}(\delta_{I_\Lambda}(p))
 \quad\text{for all }\vec n\in\mathbb N^k, \ p\in\operatorname{dom}(\delta_{I_\Lambda}).
\]
At $k=0$ this means computability of the sole induced leaf map.

\begin{proposition}[Individual realizability does not imply uniform realizability]\label{prop:individual-not-uniform}
There exists a computational problem $\mathcal P=(\Xi,\Omega,(\mathcal M,d_\mathcal M),\Lambda)$ with $\Omega=\mathcal C$ and $\mathcal M=\mathcal C\times\mathcal C$,
equipped with their standard topologies and representations, where
\[
 d_\mathcal M((u,v),(u',v')):=\max\{d_{\mathcal C}(u,u'),d_{\mathcal C}(v,v')\},
\]
and with a countable nonempty evaluation set, such that
	\begin{enumerate}
		\item The target $\Xi$ is continuous, and
		\[
		\SCIG(\mathcal P)=1,\quad
		\operatorname{rank}_{\mathrm W}(\Phi)=
		\operatorname{SCI}_{\mathrm{name}}(\mathcal P)=\infty.
		\]
		
		\item The problem admits a height $1$ raw tower $(\Gamma_n)_{n\in\mathbb N}$
		and a height $2$ raw tower $(\Delta_{n,m})_{n,m\in\mathbb N}$ satisfying
		\[
			\Gamma_n=\Delta_{n,n}.
		\]
		Every leaf admits a fixed finite nonempty query policy, and every
		induced leaf map on $I_\Lambda$ is Type-2 computable. All required metric limits converge uniformly over $\Omega$.
		
		\item Every intermediate map
		\[
			\Delta_n:=\lim_{m\to\infty}\Delta_{n,m}
		\]
		is individually Type-2 computable with respect to the specified
		representations, but admits no general-algorithm query policy.
		
		\item Neither indexed leaf family is jointly Type-2 computable.
	\end{enumerate}
\end{proposition}

\begin{proof}
	Let
	\[
	\Lambda:=\{\pi_j:j\in\mathbb N_0\},\quad
	\pi_j(x):=x(j),
	\]
	viewing the values as elements of $\mathbb C$, and fix the enumeration
	\[
	f_{j+1}:=\pi_j \quad(j\in\mathbb N_0).
	\]
	
	The map
	\[
	\operatorname{Ev}_\Lambda: \mathcal C \to I_\Lambda
	\]
	is a computable isomorphism. Indeed, bits can uniformly be converted into complex Cauchy names. Conversely, a complex Cauchy name of a promised bit determines that bit by approximation to error less than $1/4$ and comparison of the rational real part with $1/2$.
	
	For every $k\in\mathbb N_0$, define
	\[
		S_k:=\bigl\{ u\in \mathcal C:\;
		\exists q\in \operatorname{dom}(\delta_{\mathcal M} \circ L_k)
		\ \exists v\in \mathcal C\,\bigl(
		q\text{ is computable}
		\ \land\
		\delta_{\mathcal M}(L_k(q))=(u,v) \bigr) \bigr\}.
	\]
	Each $S_k$ is countable: there are only countably many computable $q\in\mathcal B$, and each such $q$ determines at most one output pair.
	Consequently
	\[
		S_\infty:=\bigcup_{k\in\mathbb N_0}S_k
	\]
	is countable. Choose
	\[
		a\in \mathcal C\setminus S_\infty.
	\]
	Every computable $u\in \mathcal C$ belongs to $S_0$, since $(u,0^\omega)$
	has a computable $\delta_M$-name and $L_0=\operatorname{id}_{\mathcal B}$.
	Hence $a$ is not computable.
	
	For $z\in C$ and $n\in\mathbb N$, define
	\[
		z^{[n]}(j):=
	\begin{cases}
		z(j),&j<n,\\
		0,&j\ge n.
	\end{cases}
	\]
	Then
	\[
		d_{\mathcal C}(z^{[n]},z)\le 2^{-n}.
	\]
	Define
	\[
		\Xi(x):=(a,x),
	\]
	and, for $n,m\in\mathbb N$,
	\[
		\Gamma_n(x):=(a^{[n]},x^{[n]}),\quad
		\Delta_{n,m}(x):=(a^{[n]},x^{[m]}),\quad
		\Delta_n(x):=(a^{[n]},x).
	\]
	The evaluations separate inputs, so these data satisfy the
	consistency condition. Moreover,
	\[
		d_{\mathcal M}(\Xi(x),\Xi(y))=d_{\mathcal C}(x,y),
	\]
	and therefore $\Xi$ is continuous.
	
	The constant policies
	\[
		Q_{\Gamma_n}(x):=\{\pi_0,\ldots,\pi_{n-1}\},\quad
		Q_{\Delta_{n,m}}(x):=\{\pi_0,\ldots,\pi_{m-1}\}
	\]
	witness both general-algorithm conditions. They are finite and nonempty because $n,m\ge1$.
	
	For each fixed $n$, the finite sequence $a\upharpoonright n$ can be hard-coded. Thus every fixed $\Gamma_n$, $\Delta_{n,m}$ and $\Delta_n$ is Type-2 computable on $\mathcal C$. The computable isomorphism $\operatorname{Ev}_\Lambda$ gives the same conclusion for their induced maps on $I_\Lambda$.
	
	The estimates
	\[
		\sup_{x\in \mathcal C} d_{\mathcal M}(\Gamma_n(x),\Xi(x)) \le 2^{-n},
	\]
	\[
		\sup_{x\in \mathcal C} d_{\mathcal M}(\Delta_{n,m}(x),\Delta_n(x)) \le 2^{-m},
	\]
	and
	\[
		\sup_{x\in \mathcal C} d_{\mathcal M}(\Delta_n(x),\Xi(x)) \le 2^{-n}
	\]
	prove all asserted uniform convergence statements. In particular,
	\[
		\Xi= \lim_{n\to\infty}\Gamma_n = \lim_{n\to\infty}\lim_{m\to\infty}\Delta_{n,m},\quad
		\Gamma_n=\Delta_{n,n}.
	\]
	Hence $\SCIG(\mathcal P)\le1$.
	
	Neither $\Xi$ nor any $\Delta_n$ is a general algorithm. Indeed, fix any such map $h$ and suppose that it admits a general-algorithm query policy. For any $x\in \mathcal C$, the finite set $Q_h(x)$ omits some coordinate query $\pi_j$. Let $y\in \mathcal C$ agree
	with $x$ everywhere except at $j$, where $y(j)=1-x(j)$. Then $x$ and $y$ agree on $Q_h(x)$, whereas the second coordinates of $h(x)$ and $h(y)$ differ. This contradicts finite-information dependence. Consequently
	\[
		\SCIG(\mathcal P)=1,
	\]
	and every intermediate map $\Delta_n$ fails to be a general algorithm.
	
	Let $\Phi:I_\Lambda\to \mathcal M$ be the induced target. Suppose, towards a
	contradiction, that
	\[
		\Phi\le_{\mathrm W}L_k
	\]
	for some $k\in\mathbb N_0$. By \Cref{thm:NTNF}, there exists a partial computable
	\[
		K:\subseteq\mathcal B\to\mathcal B \quad\text{with}\qquad L_k\circ K\vdash\Phi.
	\]
	Choose a computable $\delta_{I_\Lambda}$-name $r_0$ of $\operatorname{Ev}_\Lambda(0^\omega)$. The realizer condition implies that
	\[
		q:=K(r_0) \in\operatorname{dom}(\delta_{\mathcal M} \circ L_k).
	\]
	Since $r_0$ and $K$ are computable, $q$ is computable, and
	\[
		\delta_{\mathcal M}(L_k(q))= \Phi(\operatorname{Ev}_\Lambda(0^\omega))=
		(a,0^\omega).
	\]
	Thus $a\in S_k$, contradicting its choice. Therefore
	\[
		\operatorname{rank}_{\mathrm W}(\Phi)=\infty,
	\]
	and \Cref{cor:nameRankId} gives
	\[
		\operatorname{SCI}_{\mathrm{name}}(\mathcal P)=\infty.
	\]
	
	Finally, a jointly computable evaluator for $(n,x)\mapsto\Gamma_n(x)$, applied to $n=j+1$ and $x=0^\omega$, would compute
	\[
		a(j)=\bigl(\operatorname{pr}_1\Gamma_{j+1}(0^\omega)\bigr)(j)
	\]
	uniformly in $j\in\mathbb N_0$, contradicting noncomputability of $a$. Here $\operatorname{pr}_1(u,v):=u$ denotes the first product projection.
	
	A jointly computable evaluator for $\Delta_{n,m}$ would likewise compute $a(j)$ by taking $n=j+1$, $m=1$ and $x=0^\omega$. The computable evaluation-table isomorphism transfers these nonuniformity conclusions to the induced leaf families on $I_\Lambda$.
\end{proof}

Related separations using constant outputs can also be found in \cite[\S~7, Thms.~53-55]{Wyeth26}.

For maps into computable metric spaces, a pointwise-limit normal form generally requires a multivalued sequence-producing map, see \cite[Cor.~5.4 and the discussion following Cor.~5.5]{brattka18}. The following elementary example explains why single-valued
extensional approximants cannot always be required.

\begin{example}[Point-valued towers and name towers differ]
	Take Cauchy names on $\mathbb R$ and the discrete representation of $2$. Now put $f=\mathbf{1}_{\{0\}} : \mathbb R \to 2$. As well known, every Type-2 computable map $\mathbb R \to 2$ is continuous, hence constant because $\mathbb R$ is connected. Furthermore, every finite iterated pointwise limit of such constant maps remains constant. Thus no finite tower of computable, extensional $2$-valued approximants computes $f$. But $f$ has name-tower height $1$ as nonzeroness is semidecidable from a Cauchy name. Here height $0$ is impossible as $f$ is clearly discontinuous.
\end{example}

If the permitted point approximants instead take values in $\mathbb R$, the jointly computable functions
\[
	g_n(x):=\max\{0,1-n|x|\}
\]
converge pointwise to $\mathbf 1_{\{0\}}(x)$. These approximants need not take values in $2$.
Thus the preceding obstruction concerns the prescribed codomain of every point approximant; compare to this \cite[Thm.~16, Thm.~50 and Thm.~51]{Wyeth26}; the latter real-domain
characterizations use the strengthened semimeasurability convention adopted in their Section~6.

\section{Outlook: Philosophical-logical aspects of the abstraction attempt the SCI is aiming at}
The preceding results raise a methodological question: which additional data must be fixed for an abstract limit-height invariant to support effective comparisons? On the syntactic side, derivability is defined relative to a deductive calculus; on the semantic side, logical consequence is defined relative to a class of interpretations and a satisfaction relation, see e.g. \cite{quine2000two,Tarski02,beall2000logical,shapiro1998logical}. The moral for the present context is not that the SCI is purely philosophical, but that \textit{the informativeness of any classification depends on making its background commitments explicit}. Our goal was to do this in a way that stays mathematically sharp while remaining neutral with respect to foundational preferences.\\

\textbf{Acknowledgments } The author thanks Arno Pauly, Vasco Brattka, Matthias Schröder and Patrick Uftring for inspiring discussions around the different facets of this work. Furthermore the author is very thankful for the detailed reading and many helpful suggestions given by the two anonymous reviewers.

\textbf{Statement } During the preparation of this work the author used UniBwM-ChatGPT5.2 in order to improve the language in the abstract and introduction. After using this tool, the author reviewed and edited the content as needed and takes full responsibility for the content of the published article. UniBwM-ChatGPT6.1-Terra was then used to proofread the final version of the preprint, where the author takes full responsibility for the logical corrections needed in sequence.

\textbf{Conflict of interest } The author declares no conflict of interest.

\printbibliography

\end{document}